\documentclass[12pt, reqno]{amsart}
\usepackage{amsmath,amssymb,amsthm}

\usepackage{amsmath, amssymb}
\usepackage{amsthm, amsfonts, mathrsfs}
\usepackage{mathptmx}
\usepackage{fullpage}
\usepackage{amsfonts,graphicx}
\usepackage{ifluatex}
\usepackage{hyperref}
\numberwithin{equation}{section}

\newcommand{\NN}{\mathbb{N}}

\newcommand{\RR}{\mathbb{R}}

\newcommand{\EE}{\varepsilon}

\newcommand{\Div}{\textnormal{div}}
\newcommand{\curl}{\textnormal{curl }}

\newcommand{\supp}{\textnormal{supp }}
\newcommand{\Lip}{\textnormal{Lip}}
\newcommand{\Exp}{\textnormal{exp}}
\newcommand{\Log}{\textnormal{log}}
\newcommand{\loc}{\textnormal{loc}}

\newtheorem{Theo}{Theorem}[section]

\theoremstyle{plain}
\theoremstyle{definition}
\newtheorem{defi}[Theo]{Definition}

\theoremstyle{remark}

\newtheorem*{rema*}{Remark}
\newtheorem*{remas}{Remarks}

\date{}
\theoremstyle{definition}
\newtheorem{Definition}{Definition}[section]
\theoremstyle{Proposition}
\newtheorem{Proposition}{Proposition}[section]
\theoremstyle{Theorem}
\newtheorem{Theorem}{Theorem}[section]
\theoremstyle{Lemma}
\newtheorem{Lemma}{Lemma}[section]
\theoremstyle{Corollary}
\newtheorem{Corollary}{Corollary}[section]
\newtheorem{Remark}{Remark}[section]
\begin{document}

\author[O. Melkemi]{Oussama Melkemi}
\address{LEDPA, Universit\'e de Batna --2--\\ Facult\'e des Math\'ematiques et Informatique\\ D\'epartement de Math\'ematiques\\ 05000 Batna Alg\'erie}
\email{ou.melkemi@univ-batna2.dz}

\author[M. Zerguine]{Mohamed Zerguine}
\address{LEDPA, Universit\'e de Batna --2--\\ Facult\'e des Math\'ematiques et Informatique\\ D\'epartement de Math\'ematiques\\ 05000 Batna Alg\'erie}
\email{m.zerguine@univ-batna2.dz}

\title[Inviscid limit]
{Local persistence of geometric structures of the inviscid nonlinear Boussinesq system}

\maketitle

\begin{abstract}
Inspired by the recently published paper \cite{Hassainia-Hmidi}, the current paper investigates the local well-posedness for the generalized $2d-$Boussinesq system in the setting of regular/singular vortex patch. Under the condition that the initial vorticity $\omega_{0}={\bf 1}_{D_0}$, with $\partial D_0$ is a Jordan curve with a H\"older regularity $C^{1+\EE},\;0<\EE<1$ and the density is a smooth function, we show that the velocity vector field is locally well-posed and we also establish local-in-time regularity persistence of the advected patch. Although, in the case of the singular patch, the analysis is rather complicated due to the coupling phenomena of the system and the structure of the patch. For this purpose, we must assume that the initial nonlinear term is constant around the singular part of the patch boundary.
\end{abstract}

\tableofcontents

\section{Introduction}

\hspace{0.5cm}We consider the inviscid nonlinear Boussinesq equation in the incompressible regime, in which the density differences in fluids are usually small and their influence on the inertia of a fluid can often be neglected. In the ocean potential density variations are mostly smaller than $3/1000$ and in the atmosphere variations of potential density are typically small in the troposphere. The small density differences are however important when considering the buoyancy of fluid volumes.  The Boussinesq approximation consists in neglecting density differences in the equations except if they are multiplied by the gravity $g$ (in our case $G(\theta))$, which is usually much bigger than the vertical accelerations within the fluid, see \cite{Achim-Wirth}. The state of the fluid is described by the following set of equations.

\begin{equation}\label{Eq1}
\left\{ \begin{array}{ll}
\partial_{t}v+v\cdot\nabla v+\nabla p=G(\theta) & \textrm{if $(t,x)\in \RR_+\times\RR^2$,}\\
 \partial_{t}\theta+v\cdot\nabla \theta=0 & \textrm{if $(t,x)\in \RR_+\times\RR^2$,}\\
\Div v=0, &\\
({v},{\theta})_{| t=0}=({v}_0,{\theta}_0),\tag{NB}
\end{array} \right.
\end{equation}
where, $v(t,x)\in\RR^2$ refers to the velocity vector field localized in $x\in\RR^2$ at a time $t$, $\rho(t,x)\in\RR^{\star}_{+}$ stands for the mass density in the modeling of geophysical fluids, and $p(t,x)\in\RR$ is the force of the internal pressure which acts to enforce the incompressibility constraint $\Div v=0$ and it may be determined in terms of $v$ and $\theta$ using the Calder\'on-Zygmund transform
\begin{equation*}
p\equiv-\sum_{i,j=1}^{2}\mathfrak{R}_i\mathfrak{R}_j(v^iv^j)+\sum_{i=1}^{2}\mathfrak{L}_i(G_i(\theta))\triangleq p_v+p_\theta,
\end{equation*}
where $\mathfrak{R}_i=\frac{\partial_i}{\sqrt{-\Delta}}$ refers to the Riesz's operator, and $\mathfrak{L}_i=\frac{\partial_i}{\Delta}$ is a differential operator of order $-1$.

\hspace{0.5cm} In the general case, the action of the buoyancy forces is given  by the following vector-valued function $G(\theta)=(G_1(\theta),G_2(\theta))$ expressed by $G(\theta)=G_1(\theta)\vec{e}_1+G_2(\theta)\vec{e}_2$ and satisfies $G\in C^3(\RR^2)$ and $G(0)=(0,0)$, with $\vec e_1=(1,0)$ and $\vec e_2=(0,1)$.

\hspace*{0.5 cm}The authors S. Angenent, S. Hacker, and A. Tannenbaum in \cite{Angennet-Haker-Tannenbaum} claimed that such a system has found some applications in the optimal mass transport problem. Recently, Y. Brenier \cite{Brenier} developed this topic and established a close relation between optimal transport theory and classical convection theory for geophysical flows modeled by \eqref{Eq1}.

\hspace*{0.5 cm}Clearly, the system \eqref{Eq1} generalizes the classical inviscid Boussinesq system with $G_1(\theta) = 0$ and $G_2(\theta)=\theta$, which is given by the following system.

\begin{equation}\label{Eq2}
\left\{ \begin{array}{ll}
\partial_{t}v+v\cdot\nabla v+\nabla p=\theta\vec e_2 & \textrm{if $(t,x)\in \RR_+\times\RR^2$,}\\
 \partial_{t}\theta+v\cdot\nabla \theta=0 & \textrm{if $(t,x)\in \RR_+\times\RR^2$,}\\
\Div v=0, &\\
({v},{\theta})_{| t=0}=({v}_0,{\theta}_0).
\end{array} \right.\tag{EB}
\end{equation}
Let us note that the mathematical importance of the system \eqref{Eq1} and so \eqref{Eq2} is not only restricted to the motion of geophysical fluids but also to the formal resemblance with three-dimensional axisymmetric swirling flows. It can be shown that the solution develops singularities at time $t$ is related at the simultaneous blow-up of $\nabla \theta$ and the vorticity in $L^1_{t}L^\infty$, see \cite{Weinan-Shu}. Unfortunately, determining whether these quantities actually blow-up seems difficult as addressing a similar problem for $3d-$Euler incompressible system.

\hspace{0,5 cm} To make the presentation more convenient, we embark with a particular case, when $\theta\equiv\theta_0$ is constant, then \eqref{Eq2} reduces to the well-known Euler equations of the type
\begin{equation}\label{Eq3}
\left\{ \begin{array}{ll}
\partial_{t}v+v\cdot\nabla v+\nabla p=0 & \textrm{if $(t,x)\in \RR_+\times\RR^2$,}\\
\Div v=0, &\\
v_{| t=0}={v}_0.
\end{array} \right.\tag{E}
\end{equation}
A substantial literature regarding the well-posedness of the system \eqref{Eq3} has been produced intensively. More details and references related to \eqref{Eq3} and \eqref{Eq2}, we refer to \cite{Chae,Chae-Kim-Nam-1,Chae-Kim-Nam-2,Hmidi-Keraani,Kato-Ponce,Pak-Park,Vishik}. Notably, the classical regularity following Kato for the local well-posedness in time of \eqref{Eq3} requires that initial velocity in $H^s(\RR^N)$, with $s>\frac{N}{2}+1$ and the solutions remain smooth in any dimension \cite{Kato-Ponce}, while for planar motion or axisymmetric flows without swirl the classical solutions are global in time, yet singularities development in finite time of such solutions are an open problem. The regularity of the solutions of Euler equations has a close link with the vorticity dynamics. This latter is denoted by $\omega=\curl v$ and defined as a skew-matrix with entries
\begin{equation*}
\omega_{i,j}=\partial_j v^i-\partial_i v^j,\; 1\le i,j\le N.
\end{equation*}
A blow-up vorticity criterion for Kato's solutions following Beale-Kato-Majda \cite{Beale-Kato-Majda} reads as follows: if $T^\star$ is the maximal lifespan time then we have
\begin{equation}\label{Eq:3}
T^\star<\infty\Leftrightarrow\int_{0}^{T^\star}\|\omega(\tau)\|_{L^\infty}d\tau=\infty.
\end{equation}
In particular, for $N=2$, the vorticity can be identified as a scalar function of the type $\omega=\partial_2 v^1-\partial_1 v^1$ which evolves the following nonlinear transport equation
\begin{equation}\label{Eq:4}
\partial_{t}\omega+v\cdot\nabla\omega=0,
\end{equation}
where permits us to recover the velocity via Biot-Savart law that is to say
$$ v=\nabla^\perp\Delta^{-1}\omega,\; \nabla^\perp=(-\partial_2,\partial_1).$$
This shows that Euler equations have a Hamiltonian structure and in turn provides an infinity of conservation laws as $\|\omega(t)\|_{L^p}=\|\omega_0\|_{L^p}$ for all $p\in[1,\infty]$. Ergo, in light of \eqref{Eq:3} the Kato's solutions are globally well-posed in time. Similarly, Yudovich explored in \cite{Yudovich} this family of conservation laws, on one hand, to weaken the hyperbolic regularity, on the other hand, to formulate a new kind of weak solutions, showing that if $\omega_0\in L^1 \cap L^\infty$, the system \eqref{Eq3} admits a unique global solution. Even though, the velocity vector field loses its Lipschitzian regularity through the time and belongs to the well-known Log-Lipschitz class in short $LL$, while the involving flow $\Psi$ is only a planar homeomorphism. A subclass of Yudovitch's one encompasses the so-called vortex patches, that is $\omega_0={\bf 1}_{D_0}$ is uniformly distributed over a bounded planar domain $D_0\subset\RR^2$ which given in \eqref{Eq:4} is preserved through the time, meaning that $\omega(t)={\bf 1}_{D_t}$, with $D_t=\Psi(t,D_0)$ is the patch that moves with the flow, here we recall that the flow $\Psi$ is the unique solution of the integral equation
$$\Psi(t,x) = x + \int_0^t v(\tau,\Psi(\tau,x))d\tau.$$

The intrinsic conundrum here is regarding the regularity of the boundary evolution patch $D_t$. If we think to apply the Yudovitch's theory, we can not get any promises, since, $\Psi(t,\cdot)$ is degenerating in time, that is to say, $\Psi(t,\cdot)\in C^{e^{-\|\omega_0\|_{L^1\cap L^\infty} t}}$.

\hspace*{0,5 cm}The first systematic work in rigorous mathematics dates back to Chemin \cite{Chemin2}, see also P. Serfati in \cite{Serfati} and A. Berttozi and P. Constantin in \cite{Bertozi} which they mitigated the proof via some modifications of geometric type. The Chemin's formalism claims that when the boundary $\partial D_0$ is a Jordan curve part of $C^{1+\EE}-$class, with $0<\EE<1,$ then the regularity of $\partial D_t$ is shown to be retained over the time. The backbone of his paradigm is deeply based on the so-called logarithmic estimate, especially to ensconce that the velocity to be Lipschitz requires to invoke the striated regularity $\partial_{X_t}\omega$ in H\"older spaces of negative $C^{\EE-1}$. The specific family $X=(X_{t,\lambda})$ is opt for further advantages, it is non-degenerate, being tangential to $\partial D_t$ and each component $X_{t,\lambda}$ is defined as the push-forward of $X_{0,\lambda}$ by the flow $\Psi(t,\cdot)$ which in turns satisfies a transport equation of the form
\begin{equation}\label{Eq:5}
(\partial_t +v\cdot\nabla)X_{t,\lambda}=\partial_{X_{t,\lambda}}v.
\end{equation}
Furthermore, another advantage of this family lies in its commutation with the transport operator $\partial_t+v\cdot\nabla$ in the following way,
\begin{equation}\label{Eq:6}
(\partial_t +v\cdot\nabla)\partial_{X_{t,\lambda}}\omega=0.
\end{equation}

This latter provides for us a holistic view about the evolution of the tangential regularity of the vorticity, which in turns is the keystone for the analysis of the vortex patches topic. We point out that the Chemin's formalism is not restricted to the usual patches, but encloses the so-called generalized vortex patches. It even gets a more precise result for patches with singular boundary by showing that the regular part of the initial boundary propagates with the same regularity without being affected by the singular part which by the reversibility of the problem cannot be smoothed out by the dynamics and becomes better than $C^1$. In addition, the velocity vector field $v$ is Lipschitz far from the singular set and may undergo a blowup behavior near this set with a rate bounded by the logarithm of the distance from the singular set. Other connected subjects in differents situations for several systems can be found in \cite{Danchin1,Danchin3,Depauw,Dutrifoy,Hmidi-1,Hmidi-2,Hmidi-Zerguine,Meddour,Meddour-Zerguine,Zerguine} and the references therein.

\hspace*{0,5cm}In the matter of the local/global topic of the nonlinear Boussinesq system \eqref{Eq2}, with or without the dissipation regime has explored satisfactorily. We quote some of the authors' works by starting with S. Sulaiman, where she showed in \cite{Sulaiman} that \eqref{Eq1} is globally well-posed, where the density is governed by the transport-diffusion with fractional dissipation when the initial data $(v_0,\theta_0)$ belongs to the critical Besov space $B^{1+\frac{2}{p}}_{p,1} \times B^{-\alpha+1+\frac{2}{p}}_{p,1}\cap L^\infty$ and $G$ is a $C^5$ function with $G(0)=0$. Recently, G. Wu and X. Zheng proved in \cite{Wu} that \eqref{Eq1}, with the presence of the vertical dissipation in the velocity and density equations, admits a unique global solution, once $(v_0,\theta_0)\in H^1\times L^2\cap L^\infty$, with $(\partial_1\omega_0,\partial_1\theta_0,\partial_1^2\theta_0)\in (L^2)^3$ and $G \in C^2$ satisfy $G(0)=0$. For more information about the nonlinear Boussinesq system we refer the reader to the references \cite{Blanchard,Chen,Diaz,Milhaljan} and the references cited therein.

\hspace*{0,5cm}The doctrine of the regular (smooth) vortex patches for the coupled equations began in the work of F. Fanelli \cite{Fanelli}, where he treated the inhomogeneous Euler system. Afterward, T. Hmidi and the second author occupied with the Boussinesq equations \eqref{Eq2} when the density satisfies the transport-diffusion equation with a full Laplacian and showed that this latter is globally well-posed in time and also provided that the vorticity can be split into a singular part which is a vortex patch term and a regular part, which is deeply related to the smoothing effect for density $\omega(t)={\bf 1}_{\Omega_t}+\widetilde{\theta}(t)$. Lately, the second author settled an analogous global result, where the full Laplacian is replaced by the fractional one and gained a sharper result compared to Chemin's result concerning the Euler's system. In the same way, Hassainia and Hmidi in \cite{Hassainia-Hmidi} stated recently an even more accurate result on the local well-posedness problem for \eqref{Eq2} in the context of a regular/singular patch. For more related subject we refer to \cite{Danchin-Paicu,Danchin-Zhang,Meddour,Paicu-Zhu-1,Paicu-Zhu-2,Zerguine}.


\hspace*{0,5 cm}It could be interesting to derive a simalar result as in \cite{Hassainia-Hmidi} to the nonlinear Boussinesq system \eqref{Eq1}. Our first main result cares with the local well-posedness for the system \eqref{Eq1} in the case of regular vortex patch. To be precise, we will prove.

\begin{Theorem}\label{thm1}
Given $G\in C^3$ satisfy $G(0)=0$. Let $0<\epsilon<1, D_0$ be a bounded domain of $\mathbb{R}^{2}$ with $\partial D_0$ is a Jordan curve in H\"older space $ C^{1+\epsilon}$ and $v_{0}$ be a divergence-free vector field such that $\omega_{0}={\bf 1}_{D_0}$ and $\theta_{0}\in L^{2}(\RR^2)\cap C^{1+\epsilon}(\RR^2)$  with $\nabla\theta_{0} \in L^{a}(\RR^2)$ such that $ 1<a<2 $.
 Then there exists $T>0$ such that the system \eqref{Eq1} admits a unique solution  $(v,\theta) \in \big(L^{\infty}([0,T];\Lip(\RR^2)\big)^2.$
Besides, for all $0<t<T$ the regularity of $\partial D_t = \Psi(t,\partial D_0) $ persits through the time in the sense that belongs also in $C^{1+\epsilon}$.
\end{Theorem}
Let us give a few remarks related to the above theorem.
\begin{Remark}
In our context the assumption that $G$ is a smooth function $C^3$ over $\RR^2$ is necessary to ensure that $G\circ\theta_0$ belongs to $C^{1+\epsilon}$, and lower than $F\in C^5$ to that of \cite{Sulaiman}.
\end{Remark}
\begin{Remark}
Let us mention that the initial density $\theta_0 \in C^{1+\epsilon}$ for $0<\epsilon<1$, doesn't persists along the time, it requires the regularity more than Lipschitz for the velocity vector field. We shall mitigate this assumption in more general version of Theorem \ref{thm1}.
\end{Remark}
\begin{Remark} To establish a classical $L^p-$ estimate for $\omega-$equation, we shall need to estimate the composition $G'_i\circ\theta$ in $L^\infty$ space. Indeed, we rewrite the system \eqref{Eq1} under the vorticity-density formulation to obtain
\begin{equation}\label{Eqv-d}
\left\{ \begin{array}{ll}
\partial_{t}\omega+v\cdot\nabla \omega=\partial_1(G_2(\theta))-\partial_2(G_1(\theta)) & \textrm{if $(t,x)\in \RR_+\times\RR^2$,}\\
 \partial_{t}\theta+v\cdot\nabla \theta=0 & \textrm{if $(t,x)\in \RR_+\times\RR^2$,}\\
\Div v=0, &\\
({\omega},{\theta})_{| t=0}=({\omega}_0,{\theta}_0).
\end{array} \right.
\end{equation}
Consequently, for $p\in[1,\infty]$ and $t\ge0$ we have
\begin{equation*}
\|\omega(t)\|_{L^p}\le\|\omega_0\|_{L^p}+\sum_{i=1}^{2}\int_{\RR^2}\|G'_i\circ\theta(\tau)\|_{L^\infty}\|\nabla\theta(\tau)\|_{L^p}d\tau.
\end{equation*}
Luckily, $\|G'_i\circ\theta(\tau)\|_{L^\infty}$ is bounded due to the action of composition law in Besov spaces, see Theorem \ref{thm2.8} below and the fact $\theta$ is transported by the flow.
\end{Remark}


\hspace*{0,5 cm}In order to prove the previous theorem, our proof is inspired from the recent work of Hassainia and Hmidi, developed in \cite{Hassainia-Hmidi}. Concerning the system \eqref{Eq2}, with some depth modifications due to the nonlinear source term $G(\theta)$, which creates some technical difficulties. One of them reflects in the control of the Lipschitz norm of the velocity. If we apply the directional derivative $\partial_X$ to the first equation of \eqref{Eqv-d}, one can get
\begin{equation*}
(\partial_t + v\cdot\nabla)\partial_X\omega=\partial_X\big((\partial_1G_2(\theta))-(\partial_2G_1(\theta)\big).
\end{equation*}
We notice that the structure of the transport equation allows us to reduce the problem in controlling $\partial_X (G_i(\theta))$ in $C^\EE$ for $i=\lbrace1,2\rbrace$ by rewriting $\partial_X\big((\partial_1G_2(\theta))-(\partial_2G_1\theta)\big)$ in terms of commutators, in other words,
$$
\partial_X\big((\partial_1G_2(\theta))-(\partial_2G_1\theta)\big)=\partial_1\big(\partial_XG_2(\theta)\big)+\big[\partial_X,\partial_1\big]G_2(\theta)-\partial_2(\partial_XG_1(\theta))-\big[\partial_X,\partial_2\big]G_1(\theta),
$$
where $\big[\partial_X,\partial_2\big]G_i(\theta)$ behaves well in the H\"older space $C^{\EE-1}$ for the reason that
$$
\big[\partial_X,\partial_j\big]G_i(\theta)=-(\partial_j X)\cdot\nabla (G_i(\theta))=-(\partial_j X)\cdot (G'_i(\theta)\nabla\theta)
$$
and the fact $G'_i(\theta)$ is bounded. Besides, the estimate $\big(\partial_XG_i(\theta)\big)$ in $C^\EE$ lies in fact that $G_i(\theta)$ is also transported by the flow, that is $\partial_tG_i+v\cdot\nabla G_i=0$.
\begin{Remark}
A direct computation of $\partial_X\big(G'_2(\theta)(\partial_1\theta)-G'_1(\theta)(\partial_2\theta)\big) $, leads to the following term
$$X_1G''_2(\theta)(\partial_1\theta)^2 + X_2G''_2(\theta)(\partial_1\theta\partial_2\theta)+\partial_{X_t}\partial_1\theta,$$
so that, we need to assume that $G$ is at least of class $C^5$.  The assumption $G\in C^3$ arises from the fact $\theta$ and $G_i\circ\theta$ are transported by the flow.
\end{Remark}
\hspace{0.5cm}
Our second task of this paper occupies with the evolution singularities in the boundary $\partial D_0$ of vortex patch $\omega_0={\bf 1} _ {D_0} $ under the condition that these singularities constitute a finite set generally denoted by $\Sigma_0$. But the situation, in this case, is more complicated because of the coupled phenomena, in particular, the nonlinear source term $G(\theta)$. To surmount, these difficulties we shall, in addition, assume that $G(\theta_0)$ is constant around $\Sigma_0$, with $\theta_0$ refers to the initial density. This remarkable property will be conserved along the trajectories due to the fact that the density $\theta$ is transported by the flow and so $G(\theta)$ is also.

\hspace{0.5cm}The second main result cares with the local well-posedness topic of the system \eqref{Eq1} in the setting of singular patch. To be precise, we will prove the following theorem.

\begin{Theorem}\label{thm2}
Let $(\EE,h,a)\in ]0,1[\times]0,e^{-1}[\times]1,2[$ and $\Sigma_0$ be a closed subset of $\RR^2$. For $D_0$ a bounded domain of $\RR^2$ whose boundary $\partial D_0$ is a Jordan curve of $C^{1+\EE}$ regularity outside $\Sigma_0$ and an initial velocity $v_0$ in free-divergence, with $\omega_0={\bf 1}_{D_0}$. Assume that $\theta_0\in L^2(\RR^2)\cap C^{\EE+1}(\RR^2)$ with  $\theta_0\in W^{1,a}(\RR^2)$ and for $i\in\lbrace1,2\rbrace, $  $G_i(\theta_0)$ is a constant in small neighborhood of $\Sigma_0$. Then, there exists $T>0$ such that \eqref{Eq1} has a unique local solution

\[
(\omega,\theta)\in L^\infty\big([0,T];L^a(\RR^2) \cap L^\infty(\RR^2)\big) \times L^\infty\big([0,T];W_{G}^{1,a}(\RR^2) \cap W_{G}^{1,\infty}(\RR^2)\footnotemark{}\big).
\]

\footnotetext{The space $W_{G}^{1,a}(\RR^2)$ is defined in general case by $W_{G}^{1,p}(\RR^2)= \lbrace \theta\in L^p(\RR^2,\RR) : \nabla (G(\theta))\in L^p(\RR^2,\RR^2) \rbrace$ for every $p\in[1,\infty]$.}

Besides, the velocity vector field $v$ is Lipschitz outside the set $\Sigma_t$, with $\Sigma_t = \Psi(t,\Sigma_0)$ in the following way
\begin{equation*}
\sup_{h\in(0,e^{-1}]}\frac{\|\nabla v(t)\|_{L^{\infty}((\Sigma_t)^{c}_{h})}}{-\log h}\in L^\infty([0,T]),
\end{equation*}
where $(\Sigma_t)^{c}_{h}=\big\{x\in\RR^2 : d(x,\Sigma_t)\ge h\big\}$. In addition, $\partial D_t = \Psi(t,\partial D_0)$ remains locally in the class $C^{1+\epsilon}$ outside $\Sigma_t$. 
\end{Theorem}
\hspace{0.5cm}A few remarks are in order.
\begin{Remark} The set of initial singularities $\Sigma_0$ isn't arbitrary, it should satisfy the following geometric property: there exist two constants $\beta>0$ and $C>0$ and a neighborhood $V_0$ of $\partial D_0$ such that for every $x\in V_{0}$
\begin{equation*}
|\nabla f(x)|\ge Cd(x,\Sigma_0)^{\beta},
\end{equation*}
where $f$ is a smooth function from $\RR^2$ into $\RR$ so that
\begin{equation*}
D_{0}=\big\{x\in\RR^2 : f(x)>0\big\},\quad \partial D_{0}=\big\{x\in\RR^2 : f(x)=0\big\}.
\end{equation*}
Meaning that the curves defining $\partial D_0$ are not tangent to one another at infinite order at the singular points.
\end{Remark}
\begin{Remark} The quantity $G_i(\theta_0)$ is constant around $\Sigma_0$, so $G_i(\theta)$ is also constant, however, around $\Sigma_t$ and so $\supp G_i(\theta)$ is included in $\Sigma_t$ so  Proposition \ref{prop4} is required in this case. On the other hand, $G_i(\theta_0)$ is constant is more general than $\theta_0$ is constant to that of \cite{Hassainia-Hmidi}.
\end{Remark}
\begin{Remark} The fact that $G_i(\theta_0)$ is constant around $\Sigma_0$ forces us to deal in what follows with it instead $\theta_0$ either in a priori estimates or existence and uniqueness topic. This is regards as a generalization of results to that \cite{Hassainia-Hmidi}.
\end{Remark}

Let us briefly outline the proof of the previous theorem. We will explore the Chemin's approach for bidimensional Euler equations \eqref{Eq3} and the recently established work by Hassainia and Hmidi in \cite{Hassainia-Hmidi} concerning the system \eqref{Eq2}. Meaning that we control the Lipschitz norm of the velocity $\|\nabla v\|_{L^\infty}$ by the striated regularity of its vorticity $\|\omega\|_{C^{\EE}(\mathcal{X}_t)}$ via a logarithmic estimate, where $\mathcal{X}_t=(X_{t,\lambda,h})_{(\lambda,h)\in \Lambda\times ]0,e^{-1}]}$ is a family of vector fields, so that each component $X_{t,\lambda,h}$ satisfies the following inhomogeneous equation
\begin{equation*}
\partial_{t}X_{t,\lambda,h}+v\cdot\nabla X_{t,\lambda,h}= \partial_{X_{t,\lambda,h}}v,
\end{equation*}
where the new subscript $h$ appears in the setting of singular patch refers to the trancate parameter around the set of singulirities. To treat the problem in the presence of this kind of singularities we must develop the two terms $\partial_1\big(G_2(\theta)\big)$ but in this case we will find some difficulties regarding the treatment of the term $\|\partial_j v\cdot\nabla\big(G_i(\theta(\tau))\big)\|_{L^p}$ which comes from the following classical $L^p$ estimates
\begin{equation*}
\|\partial_j \big(G_i(\theta(t))\big)\|_{L^{p}} \leq  \|\partial_j  \big(G_i(\theta_0)\big)\|_{L^{p}} + \int_0^t \|\partial_j v\cdot\nabla\big(G_i(\theta(\tau))\big)\|_{L^p}d\tau.
\end{equation*}
To remedy this drawback we must assume that the initial buoyancy forces $G(\theta_0)$ is constant, enclosing the singularities set. This latter is advected by the flow and also satisfying the transport equation, so that the Proposition \label{prop4} is then applicable.\\

{\bf Organization of the paper.}  Section \ref{S3} concerns the case of a smooth vortex patch. We start by some basic useful tools and definitions and a concise of Littlewood-Paley theory, where we state the cut-off operators, paradifferential calculus, and some properties of Besov spaces and particular cases. Thereafter, we undertake the preparatory part for the smooth vortex patch, and we give some of the prior estimates for the vorticity and the density. Finally, we discuss the proof of the main result in several steps. In section \ref{S4}, we return to the case of the singular patch. First, we state the suitable framework for them and we will follow the same steps as in section \ref{S3}.
\section{Regular vortex patches}\label{S3}
In this part we shall recall some tools  the so called Littlwood-paley operators and Bony's decomposition, we will also introduce some function spaces as Besov and H\"older spaces.

{\bf Notations.} During this work, we will agree some useful notations.

$\bullet$ We denote by $C$ any positive constant which changes from line to another and we shall use the notation $X\lesssim Y$ instead the notation $\exists C > 0 $ such that $X\leq C Y$ and $C_0$ is a positive constant depending on the initial data.

$\bullet$ For every $ p\in[1,\infty], \| \cdot \|_{L^p}$ denotes the $L^p$ norms.

$\bullet$ For $ u \in C^\epsilon, \|  \cdot\|_{\epsilon}$ denotes the $C^\epsilon$ norms.

$\bullet$ For $P,Q$ two operators, the commutator $\big[P,Q\big]$ is defined by $PQ-QP$.

\subsection{Preparatory and preliminaries}\label{Little-P}
We recall the Littlwood-Paley theory based on nonhomogeneous dyadic partition of unity. Let $(\chi,\varphi)\in\mathcal{D}(\RR^2)\times \mathcal{D}(\RR^2\backslash\{0\})$ be a radial cut-off functions be such that $\supp\chi\subset \{\xi\in\RR^2: |\xi|\le1\}$ and $\supp\varphi\subset\{\xi\in\RR^2: 1/2\le|\xi|\le2\}$, so that

\begin{equation*}
\chi(\xi)+\sum_{q\ge0}\varphi(2^{-q}\xi)=1.
\end{equation*}
For every $ u\in S'(\RR^2)$ we define the cut-off operators as follows,
\begin{displaymath}
\Delta_{-1}v\equiv \mathcal{F}^{-1}(\chi \small\widehat{v}),\qquad  \Delta_qv \equiv \mathcal{F}^{-1}(\phi(2^{-q}.) \small\widehat{v}),\qquad S_qv =\sum_{-1 \leq j\leq q-1} \Delta_j v,  \qquad \forall q \in \mathbb{N}.
\end{displaymath}
We can see also the cut-off operators as a Fourier multipliers.
Now we recall the Besov space in terms of the Littlewood-Paley operators.

\begin{Definition}
For $(s,p,r)\in\mathbb{R}\times[1,+\infty]^2$. The inhomogeneous Besov space $\mathrm{B}_{p,r}^s$ is the set of tempered distributions $v\in S'$ such that
$$\|v\|_{\mathrm{B}_{p,r}^s}\triangleq\big(2^{qs}\|\Delta_qv\|_{L^p}\big)_{\mathfrak{l}^r(\mathbb{Z})}<\infty.$$
\end{Definition}

\begin{remas}
We notice that :

$\bullet$ For $ s\in\mathbb{R}_{+} \char`\\ \mathbb{N}$ the H\"older space noted by $C^s $ coincides with $\mathrm{B}_{\infty,\infty}^s.$

$\bullet$ $(C^s,\|\cdot\|_{s})$ is a Banach space that coincides with the usual H\"older space $C^s$ with equivalent norms,

\begin{equation*}
\|v\|_{s}\lesssim\|v\|_{L^{\infty}}+\sup_{x\neq y}\frac{|v(x)-v(y)|}{|x-y|^s}\lesssim\|v\|_{s}.
\end{equation*}
$\bullet$ If $s\in\mathbb{N}$, the obtained space is so-called H\"older-Zygmund space and still denoted by $\mathrm{B}_{\infty,\infty}^s.$
\end{remas}

Let us now introduce the well-known Bony's decomposition, which split the product of two tempered distributions into three parts. Namely : for $u,v\in S'$
$$
uv=T_uv+T_vu+\mathcal{R}(u,v),
$$
with
$$
T_uv=\sum_q S_{q-1}u\Delta_qv,\quad\mathcal{R}(u,v)=\sum_q\Delta_q \widetilde{\Delta}_qv\quad\text{ and }\quad\widetilde{\Delta}_q=\Delta_{q-1}+\Delta_q+\Delta_{q+1}.
$$
The next result deals with the action of Bony's decomposition in the H\"older space. For more details we refer the reader to \cite{Bahouri-Chemin-Danchin}.
\begin{Lemma}\label{L}
Let $r$ be a real number. If $r<0$, the operator $T$ is continuous from $L^{\infty} \times C^r$ in $C^r$ and from $C^r \times L^{\infty}$ in $C^r.$ Furthermore, we have
$$\|T_u v\|_r + \|T_v u\|_r \leq C(r)\|u\|_{L^{\infty}} \|v\|_r.$$
If $r>0$, the operator $\mathcal{R}$ is continuous from $L^{\infty} \times C^r$ in $C^r.$ Moreover, we have
$$\|\mathcal{R}(u,v)\|_r \leq C(r) \|u\|_{L^{\infty}} \|v\|_r.$$
\end{Lemma}
The following theorem treats the action of composition law with a smooth functions in the Besov spaces and it playes a significant role in the sequel. The proof can be found in \cite{Serge}.
\begin{Theorem}\label{thm2.8}
Let $G \in C^{[s]+2}$, with $G(0)=0$ and $s\in[0,\infty]$. Assume that $\theta \in B^{s}_{p,r} \cap L^{\infty} ,$ with $(p,r) \in [1,+\infty]^2$, then $G\circ \theta \in B^{s}_{p,r}$ and satisfying
$$\|G \circ \theta\|_{B^{s}_{p,r}} \leq C(s)\sup_{|y|\leq C\|\theta\|_{L^\infty}}\|G^{[s]+2}(y)\|_{L^\infty} \|\theta\|_{B^{s}_{p,r}}.$$
\end{Theorem}
The persistence of Besov regularity for transport equations which will be useful in several situations can be reads as follows. For the proof, we refer to \cite{Bahouri-Chemin-Danchin}.
\begin{Proposition}\label{lem3}
Let $s \in ]-1,1[$ and $v$ be a smooth divergence-free vector field. Let us consider a couple of functions $(a,f)\in L^{\infty}_{\loc}(\mathbb{R},C^s) \times L^{1}_{\loc}(\mathbb{R},C^s)$ and $a_0\in C^s$ such that
\begin{equation*}
\left\{ \begin{array}{ll}
\partial_{t}a+v\cdot\nabla a=g \\
a_{| t=0}={a}_0.
\end{array} \right.
\end{equation*}
Then for all $t\geq 0$, we have
\begin{equation}\label{eq8}
\|a(t)\|_s \lesssim e^{V(t)}\Big(\|a_0\|_s + \int_0^t \|g(\tau)\|_s d\tau\Big),
\end{equation}
where $V(t)=e^{C\int_0^t (\|\nabla v(\tau)\|_{L^{\infty}})d\tau}$ with $C$ being a constant depending only on $s$.
\end{Proposition}
\subsection{Regular patch tool box}
We give some definitions and notations concerning the admissible family of vector fields and the anisotropic H\"older space. These quantities constitute the main ingredients concerning the vortex patch problem.
\begin{Definition}\label{D1}
Let $\Sigma$ be a closed set of the plane and $\epsilon$ $\in$ (0,1). Let $X=(X_{\lambda})_{(\lambda \in \Lambda)}$ be a family of vector fields. We say that this family is admissible of class $C^{\epsilon}$ outside $\Sigma$ if and only if :

$\bullet$ Regularity: $X_{\lambda}$, div $X_{\lambda} \in C^{\epsilon}$.

$\bullet$ Non degeneracy:

$$ I(\Sigma,X) \triangleq \inf_{x\notin\Sigma} \sup_{\lambda\in\Lambda} |X_{\lambda}(x)| > 0 .$$
We set
$$ \check{\|} X_{\lambda} \|_{\epsilon}\triangleq \|X_{\lambda}\|_{\epsilon} +\|\Div\text{ }X_{\lambda}\|_{\epsilon-1} $$
and
$$ N_\epsilon(\Sigma,X)  \triangleq  \sup_{\lambda\in\Lambda}\frac{\check{\|} X_{\lambda} \|_{\epsilon}}{I(\Sigma,X)}.$$

The action of the family $X_{\lambda}$ on bounded real-valued functions $u$ in the weak sense as follows:
$$ \partial_{X_{\lambda}}u \triangleq \Div\text{}(uX_{\lambda})-u\Div\text{}X_{\lambda}.$$
Now, for all $t \in [0,T] $ the transported $X_t=(X_{t,\lambda})$ of an initial family $X_0=(X_{0,\lambda})$ by the flow $\Psi ,$ is defined by
\begin{equation}\label{Eq11}
X_{t,\lambda}(x) \triangleq \big(\partial_{X_{0,\lambda}}\Psi(t)\big)\big(\Psi^{-1}(t,x)\big).
\end{equation}
\end{Definition}
The next definition deals with the concept of anisotropic H\"older space, denoted by $C^{\epsilon+k}(\Sigma,X).$
\begin{Definition}\label{D2}
Let $0<\epsilon<1,k\in \mathbb{N}$ and $\Sigma$ be a closed set of the plane. Consider a family of vector fields $X=(X_\lambda)_{\lambda}$ as in Definition \ref{D1}. We say that $v\in C^{\epsilon+k}$ the space of functions $v\in W^{k,\infty}$ such that
$$\sum_{|\alpha|\leq k} \|\partial^{\alpha}v\|_{L^{\infty}} + \sup_{ \lambda \in \Lambda} \|\partial_{X_{\lambda}} v\|_{\epsilon+k-1} < \infty $$
and we set
$$ \|v\|^{\epsilon+k}_{\Sigma,X} \triangleq N_\epsilon(\Sigma,X) \sum_{|\alpha|\leq k} \|\partial^{\alpha}v\|_{L^{\infty}} + \sup_{ \lambda \in \Lambda}\frac{\|\partial_{X_{\lambda}} v\|_{\epsilon+k-1}}{I(\Sigma,X)}.$$
\end{Definition}
The result below is a direct consequence of Lemma \ref{L}, see for instance Corollary 3.1 \cite{Hassainia-Hmidi}.
\begin{Corollary}\label{Cor1}
Let $\epsilon \in ]0,1[,X$ be a vector field belonging to $C^\epsilon$ such that $\Div X$ belonging to $C^\epsilon$ too and $g$ be a Lipschitz scalar function. Then for $i\in \{1,2\}$ we have
$$\|(\partial_i X)\cdot\nabla g\|_{\epsilon-1} \lesssim \|\nabla g\|_{L^{\infty}}(\|\Div X\|_\epsilon + \|X\|_\epsilon).$$
\end{Corollary}

Next, we state the logarithmic estimate introduced by Chemin \cite{Chemin2}, it allows us to control the Lipschitz norm of the velocity with respect to the striated regularity.
\begin{Theorem}\label{thm3}
Let $a \in (1,\infty), \epsilon \in (0,1), \Sigma$ be a closed set of the plane and $X$ be a family of vector fields as in definition \ref{D1}. Consider $\omega \in C^{\epsilon} (\Sigma,X)  \cap L^{a}.$ Let $v$ be a divergence-free vector field with vorticity $\omega$, then there exists C such that
$$\|\nabla v(t)\|_{L^{\infty} (\Sigma)} \leq C(a,\epsilon) \bigg( \|\omega(t)\|_{L^{a}} + \|\omega(t)\|_{L^{\infty}}\Log\bigg(e+\frac{\|\omega(t)\|^{\epsilon}_{\Sigma,X}}{\|\omega(t)\|_{L^{\infty}}}\bigg)\bigg).$$
\end{Theorem}
A smooth bounded domain in the plane is arounded by a closed curve which is characterized by certain geometric properties given by the following definition.  
\begin{defi}\label{Chart} Let $\EE>0$. A closed curve $\Gamma$ is said to be of class $C^{1+\EE}$, if there exists $f\in C^{1+\EE}(\RR^2)$ such that $\Sigma$ is  locally a zero set of $f$, that is there exists a neighborhood $V$ of $\Sigma$ such that
\begin{equation}\label{Reg-Om}
\Sigma=f^{-1}(\{0\})\cap V,\quad \nabla f(x)\ne 0\quad\forall x\in V.
\end{equation}   
\end{defi}  
\subsection{A priori estimates for the vorticity and density}
This subsection concerns the classical $L^p$ estimates for the vorticity and the density taking into account the effect of the buoyancy term  $G(\theta)$ which be helpful ingredients in the sequel. 
\begin{Proposition}\label{prop1}
Let $p \in [1,\infty]$ and $t \leq T$ and assume that $(v,\theta)$ is a smooth solution of the system (\ref{Eq1}) defined on the interval $[0,T]$ and $G \in C^{1}.$ Then
\begin{equation}\label{Eq13}
\|\omega(t)\|_{L^{p}} \lesssim \|\omega_{0}\|_{L^{p}} + \|\nabla \theta_{0}\|_{L^{p}}e^{CV(t)}t,
\end{equation}
\begin{equation}\label{Eq12}
\|\nabla \theta(t)\|_{L^{p}} \lesssim \|\nabla \theta_{0}\|_{L^{p}}e^{CV(t)},
\end{equation}
where $$V(t)= \int_0^t \|\nabla v(\tau)\|_{L^{\infty}}d\tau.$$
\begin{proof}
Let $j={1,2}$, applying the operator $\partial_j$ to $\theta-$equation of \eqref{Eq1}, so we have
$$ \partial_{t}\partial_{j} \theta + v\cdot\nabla\partial_{j} \theta =-(\partial_{j}v)\cdot\nabla\theta.$$
Taking the $L^{p}-$norm, using H\"older inequality and $\Div v=0$ we get for every $p \in [1,\infty ]$
$$\| \partial_{j} \theta(t)\|_{L^{p}} \leq \|\partial_{j} \theta_{0}\|_{L^{p}} + \int_0^t \|\nabla \theta(\tau)\|_{L^{p}}\|\nabla v(\tau)\|_{L^{\infty}}d\tau.$$
Next, Gronwall's estimate leading to
$$\|\nabla \theta(t)\|_{L^{p}} \lesssim \|\nabla \theta_{0}\|_{L^{p}}e^{CV(t)}.$$
For the second estimate we use the vorticity-density equation in \eqref{Eqv-d}
$$\partial_{t}\omega+v\cdot\nabla \omega=G'_2(\theta)(\partial_1\theta)-G'_1(\theta)(\partial_2\theta).$$
Again $L^{p}$ estimate gives
\begin{equation}\label{omega-bounded}
\|\omega(t)\|_{L^{p}} \leq \|\omega_{0}\|_{L^{p}}+ \int_0^t(\sum_{i=1}^2\|G'_{i}(\theta)\|_{L^{\infty}}) \|\nabla \theta(\tau)\|_{L^{p}}d\tau.
\end{equation}

Since $G$ is a $C^{1}$ function, so that

$$\|G'_{i}(\theta)\|_{L^{\infty}} \leq  \sup_{|x|\leq \|\theta(t)\|_{L^{\infty}}}|G'_{i}(x)|.$$

On the other hand, by the maximum principle we deduce $\|\theta(t)\|_{L^{\infty}} = \|\theta_{0}\|_{L^{\infty}}.$ Consequently we find out
\begin{eqnarray}\label{G'-bounded}
 \|G'_{i}(\theta)\|_{L^{\infty}} &\leq & \sup_{|x|\leq \|\theta_{0}\|_{L^{\infty}}}\|\nabla G_{i}(x)\|\\
 \nonumber &\leq & C.
\end{eqnarray}
Grouping \eqref{omega-bounded} and \eqref{G'-bounded}, we finally get
$$\|\omega(t)\|_{L^{p}} \lesssim \|\omega_{0}\|_{L^{p}} + \|\nabla \theta_{0}\|_{L^{p}}e^{CV(t)}t.$$
\end{proof}
\end{Proposition}
\subsection{A priori estimates for the striated regularity of the vorticity}
In this part we study some nice properties of the family $(X_\lambda)=(X_{t,\lambda})_{\lambda\in\Lambda}$ often constructed in \cite{Chemin2} and stated the striated regularity of the vorticity.
\begin{Lemma}\label{cor2}
There exists a constant $C$ such that for any smooth solution $(v,\theta)$ of \eqref{Eq1} on $[0,T]$ and any time dependent family of vector field $X_t$ transported by the flow of $v$, we have for all $t\in [0,T]$,
\begin{eqnarray}\label{Eq15}
I(X_0) \leq I(X_t)e^{CV(t)},\\
\label{Eq16}\|\Div X_{t,\lambda}\|_\epsilon \leq \|\Div X_{0,\lambda}\|_\epsilon e^{CV(t)},
\end{eqnarray}

\begin{eqnarray}\label{Eq17}
\check{\|}X_{t}\|_{\epsilon} + \|\partial_{X_{t,\lambda}}\omega\|_{\epsilon-1} &\leq & C\big(\check{\|}X_{0,\lambda}\|_{\epsilon}+ \|\partial_{X_{0,\lambda}}\omega_{0}\|_{\epsilon-1}+\|\theta_0\|_{\epsilon}\|\partial_{X_{0,\lambda}}\theta_0\|_{\epsilon}\big) e^{Ct} e^{CV(t)}\nonumber  \\
&\times& e^{t\|\nabla \theta_0\|_{L^\infty} e^{CV(t)}}.
\end{eqnarray}
\end{Lemma}

\begin{proof}
Applying the partial derivative $\partial_t$ to the quantity $\partial_{X_{0,\lambda}} \Psi(t,x)$ to obtain
\begin{equation}\label{eq2.9}
\left\{ \begin{array}{ll}
\partial_t\partial_{X_{0,\lambda}} \Psi(t,x)=\nabla v(t,\Psi(t,x))\partial_{X_{0,\lambda}} \Psi(t,x)\\
\partial_{X_{0,\lambda}} \Psi(0,x)=X_{0,\lambda}.
\end{array} \right.
\end{equation}
Combining Gronwall's inequality with the time inversibility of \eqref{eq2.9}, it holds
$$|X_{0,\lambda}(x)| \leq |\partial_{X_{0,\lambda}} \Psi(t,x)|e^{CV(t)}.$$
Grouping the Definition \ref{D2} and \eqref{Eq11} we get \eqref{Eq15}.

\hspace{0.5cm}For \eqref{Eq16}, using the fact that $X_{t,\lambda}$ satisfies the following equation
\begin{equation}\label{Eq19}
\partial_t X_{t,\lambda} + v \cdot \nabla X_{t,\lambda} = \partial_{X_{t,\lambda}} v.
\end{equation}
Apply the operator $\Div$ on \eqref{Eq19} to find
$$\Div(\partial_t X_{t,\lambda} + v \cdot \nabla X_{t,\lambda}) = \Div( \partial_{X_{t,\lambda}} v).$$
So $\Div v =0$ implies
$$(\partial_t + v\cdot \nabla)\Div X_{t,\lambda}=0.$$
At this stage the estimate (\ref{Eq16}) directly follows from Proposition \ref{lem3}.

\hspace{0.5cm}To estimate $X_t$ in $C^\epsilon$, again Proposition \ref{lem3} to \eqref{Eq19},  it happens
\begin{equation}\label{x}
\|X_{t,\lambda}\|_\epsilon \leq e^{CV(t)} \big( \|X_{0,\lambda}\|_\epsilon + C\int_0^t \|\partial_{X_{t,\lambda}} v(\tau)\|_\epsilon e^{-CV(\tau)}d\tau \big).
\end{equation}
Using the following inequality where its proof can be found in [Lemma 3.3.2 in \cite{Chemin2}],
\begin{equation}\label{x1}
\|\partial_{X_{t,\lambda}} v(t)\|_\epsilon \lesssim \|\nabla v(t)\|_{L^{\infty}}\check{\|}X_{t,\lambda}\|_\epsilon + \|\partial_{X_{t,\lambda}}\omega(t)\|_{\epsilon-1}.
\end{equation}
Combined \eqref{x} with \eqref{x1} to write
\begin{equation}\label{Eq20}
\check{\|}X_{t,\lambda}\|_\epsilon \leq e^{CV(t)}\bigg(\check{\|}X_{0,\lambda}\|_\epsilon + C\int_0^t e^{-CV(\tau)}\big(\|\nabla v(\tau)\|_{L^{\infty}}\check{\|}X_{\tau,\lambda}\|_\epsilon + \|\partial_{X_{\tau,\lambda}}\omega(t)\|_{\epsilon-1}\big)d\tau\bigg).
\end{equation}
\hspace{0.5cm}Now, let us move to bound the term $\|\partial_{X_{t,\lambda}}\omega\|_{\epsilon-1}$. For this purpose we make use the commutation between the operator $\partial_{X_{t,\lambda}}$ with transport operator, so that
$$ (\partial_{t}+v\cdot\nabla)\partial_{X_{t,\lambda}}\omega = \partial_{X_{t,\lambda}}\big(\partial_{1}\big(G_{2}(\theta)\big)-\partial_{2}\big(G_{1}(\theta)\big)\big).$$
Thus, Proposition \ref{lem3} provides
$$\|\partial_{X_{t,\lambda}}\omega\|_{\epsilon-1} \leq e^{CV(t)}\bigg(\|\partial_{X_{0,\lambda}}\omega_{0}\|_{\epsilon-1} + C\int_0^t e^{-CV(\tau)}\|\partial_{X_{t,\lambda}}\big(\partial_{1}\big(G_{2}(\theta)\big)-\partial_{2}\big(G_{1}(\theta)\big)\big)\|_{\epsilon-1}d\tau\bigg).$$
It follows that
 \begin{eqnarray}\label{eqinto}
\|\partial_{X_{t,\lambda}}\omega\|_{\epsilon-1} \leq e^{CV(t)}(\|\partial_{X_{0,\lambda}}\omega_{0}\|_{\epsilon-1} &+& C\int_0^t\|\partial_{X_{t,\lambda}}\big(\partial_{1}(G_{2}(\theta))\big)\|_{\epsilon-1}e^{-CV(\tau)}d\tau\\ &+& C\int_0^t\|\partial_{X_{t,\lambda}}\big(\partial_{2}(G_{1}(\theta))\big)\|_{\epsilon-1}e^{-CV(\tau)}d\tau).\nonumber
 \end{eqnarray}
To control the term $\|\partial_{X_{t,\lambda}}(\partial_{1}(G_{2}(\theta))\|_{\epsilon-1}$  we use the fact that
$$\partial_{X_{t,\lambda}}(\partial_{1}(G_{2}(\theta))= \partial_1\big(\partial_{X_{t,\lambda}}G_2(\theta)\big)+\big[\partial_{X_{t,\lambda}},\partial_1\big]G_2(\theta). $$
In such a way we have
\begin{equation*}
\|\partial_{X_{t,\lambda}}(\partial_{1}(G_{2}(\theta))\|_{\epsilon-1}  \leq \|\partial_{1}(\partial_{X_{t,\lambda}}(G_{2}(\theta))\|_{\epsilon-1} + \|\big[\partial_{X_{t,\lambda}},\partial_1\big]G_2(\theta)\|_{\epsilon-1}.
\end{equation*}
On the other hand we can see that
\begin{equation}\label{com:1}
\big[\partial_X,\partial_1\big]G_2(\theta)=-\big(\partial_1 X_{t,\lambda}\big)\cdot\nabla G_2(\theta).
\end{equation}
Accordingly finds out
$$\|\partial_{X_{t,\lambda}}(\partial_{1}(G_{2}(\theta))\|_{\epsilon-1} \leq \|\partial_{X_{t,\lambda}}(G_{2}(\theta))\|_{\epsilon} + \|\partial_{1}(X_{t,\lambda})\cdot\nabla(G_{2}(\theta))\|_{\epsilon-1}.$$

In view of Corollary \ref{Cor1} and the fact $G\circ\theta \in L^\infty$ we get
\begin{equation}\label{Eq23}
\|\partial_{X_{t,\lambda}}(\partial_{1}(G_{2}(\theta))\|_{\epsilon-1} \lesssim
\|\partial_{X_{t,\lambda}}(G_{2}(\theta))\|_{\epsilon} + \|\nabla\theta(\tau)\|_{L^{\infty}}\check{\|}X_{t,\lambda}\|_{\epsilon}.
\end{equation}
To control $\|\partial_{X_{t,\lambda}}(G_{2}(\theta))\|_{\epsilon}$, we employ that
$\theta$ is transported by the flow, so that $G_{i}(\theta(t,x))$ is also. Indeed, for $i\in\lbrace 1,2 \rbrace$ we have
$$\partial_t G_{i}(\theta(t,x))+ v\cdot\nabla G_{i}(\theta(t,x)) = G'_{i}(\theta(t,x))(\partial_t \theta + v\cdot\nabla \theta)=0.$$

Exploit the commutation between $\partial_{X_{t,\lambda}}$ with transport operator $(\partial_t+v\cdot\nabla)$ and applying Proposition \ref{lem3} we get
$$\|\partial_{X_{t,\lambda}}(G_{i}(\theta))\|_{\epsilon}\leq \|\partial_{X_{0,\lambda}}(G_{i}(\theta_0))\|_{\epsilon}e^{CV(t)}.$$

On the other hand, we can easily check that $\partial_{X_{0,\lambda}}(G_i(\theta_0))= G'_i(\theta_0)\partial_{X_{0,\lambda}}\theta_0$, so the fact that $C^\epsilon$ is an algebra ensures us to write
\begin{eqnarray}\label{Eqalg}
\|\partial_{X_{0,\lambda}}(G_{i}(\theta_0))\|_{\epsilon} &\leq & \|G'_i(\theta_0)\|_{\epsilon}\|\partial_{X_{0,\lambda}}\theta_0\|_{\epsilon}\\ \nonumber
&\lesssim & \|\theta_0\|_{\epsilon}\|\partial_{X_{0,\lambda}}\theta_0\|_{\epsilon}.
\end{eqnarray}
We have used in the last inequality Theorem \ref{thm2.8}. Plug \eqref{Eqalg} into \eqref{Eq23} to conclude that

\begin{equation}\label{eqq}
\|\partial_{X_{t,\lambda}}(\partial_{1}(G_{i}(\theta))\|_{\epsilon-1} \lesssim \|\theta_0\|_{\epsilon}\|\partial_{X_{0,\lambda}}\theta_0\|_{\epsilon}e^{CV(t)} + \|\nabla\theta(\tau)\|_{L^{\infty}}\check{\|}X_{t,\lambda}\|_{\epsilon}.
\end{equation}
Inserting \eqref{eqq} into \eqref{eqinto} one gets
$$
\|\partial_{X_{t,\lambda}}\omega\|_{\epsilon-1} \lesssim e^{CV(t)}(\|\partial_{X_{0,\lambda}}\omega_{0}\|_{\epsilon-1} + \|\theta_0\|_{\epsilon}\|\partial_{X_{0,\lambda}}\theta_0\|_{\epsilon}t + \int_0^t\|\nabla\theta(\tau)\|_{L^{\infty}}
\check{\|}X_{t,\lambda}\|_{\epsilon}e^{-CV(\tau)}d\tau).
$$
Setting $$\mathfrak{A}(t) \triangleq (\|\partial_{X_{t,\lambda}}\omega\|_{\epsilon-1}+\check{\|}X_{t,\lambda}\|_{\epsilon})e^{-CV(t)}.$$
By virtue of the last estimate and (\ref{Eq20}) we find
$$\mathfrak{A}(t) \lesssim \mathfrak{A}(0) + \|\theta_0\|_{\epsilon}\|\partial_{X_{0,\lambda}}\theta_0\|_{\epsilon}t + \int_0^t (\|\nabla \theta\|_{L^{\infty}}+\|\nabla v\|_{L^{\infty}}+1)\mathfrak{A}(\tau)d\tau.$$
Gronwall's inequality ensures that
$$\mathfrak{A}(t) \lesssim (\mathfrak{A}(0) + \|\theta_0\|_{\epsilon}\|\partial_{X_{0,\lambda}}\theta_0\|_{\epsilon})e^{C \int_0^t (\|\nabla \theta\|_{L^{\infty}}+\|\nabla v\|_{L^{\infty}}+1)d\tau}.$$
Estimate \eqref{Eq12} completes the proof.
\end{proof}

\subsection{Lipschitz bound of the velocity}This paragraph addresses to the proof of the Lipschitz norm of the velocity $\|\nabla v(t)\|_{L^\infty}$ locally in time, which considered as the core part in the Theorem \ref{thm1}.
\begin{Proposition}\label{Lip-v}
Let $v$ and $\theta$ be smooth solutions of the system (\ref{Eq1}) defined on the time interval $[0,T^{\star}[$. Then there exists $T_{0}$ between $0$ and   $\text{}$ $T^{\star}$ such for all $t \leq$ $T_{0}$ we have:
\begin{equation}
 \|\nabla v(t)\|_{L^{\infty}}\leq M_{0}.
\end{equation}
\end{Proposition}
\begin{proof}
For simplicity we set $$\mathfrak{B}(t)= e^{Ct}e^{CV(t)}e^{Ct\|\nabla \theta_0\|_{L^{\infty}}e^{CV(t)}}.$$
In particular, from (\ref{Eq17}) we have
$$ \check{\|}X_{t,\lambda}\|_{\epsilon} \leq C(\mathfrak{A}(0)+\|\theta_0\|_{\epsilon}\|\partial_{X_{0,\lambda}}\theta_0\|_{\epsilon})\mathfrak{B}(t). $$
Multiplying the last inequality by $\|\omega(t)\|_{L^{\infty}}$ and using the \eqref{Eq12} to write
$$ \|\omega(t)\|_{L^{\infty}}\check{\|}X_{t,\lambda}\|_{\epsilon} \leq C(\|\omega_0\|_{L^{\infty}}+\|\nabla\theta_0\|_{L^{\infty}}te^{Ct})(\mathfrak{A}(0)+\|\theta_0\|_{\epsilon}\|\partial_{X_{0,\lambda}}\theta_0\|_{\epsilon})\mathfrak{B}(t).$$
Consequently
$$\|\omega(t)\|_{L^{\infty}}\check{\|}X_{t,\lambda}\|_{\epsilon} \leq C(\|\omega_0\|_{L^{\infty}}+1)(\mathfrak{A}(0)+\|\theta_0\|_{\epsilon}\|\partial_{X_{0,\lambda}}\theta_0\|_{\epsilon})\mathfrak{B}(t).$$
Hence
$$\|\partial_{X_{t,\lambda}}\omega(t)\|_{\epsilon-1}+\|\omega(t)\|_{L^{\infty}}\check{\|}X_{t,\lambda}\|_{\epsilon} \leq C(\|\omega_0\|_{L^{\infty}}+1)(\mathfrak{A}(0)+\|\theta_0\|_{\epsilon}\|\partial_{X_{0,\lambda}}\theta_0\|_{\epsilon})\mathfrak{B}(t).$$
According to \eqref{Eq15} and Definition \ref{D2} one gets
\begin{equation}\label{Eqz}
\|\omega(t)\|^{\epsilon}_{X_t} \leq M_0\mathfrak{B}(t),
\end{equation}
in view of Theorem \ref{thm3} and the monotonicity of the map $u\mapsto u\Log(e+\frac{b}{u})$ we obtain
$$\|\nabla v(t)\|_{L^{\infty}} \leq C_{\epsilon}\big(\|\omega_{0}\|_{L^{a}\cap L^{\infty}} +t\|\nabla \theta_0\|_{L^{a} \cap L^{\infty}}e^{CV(t)}\big)\Log \big(e+\frac{\|\omega(t)\|^{\epsilon}_{X_t}}{\|\omega_0\|_{L^{\infty}}}\big).$$
Using (\ref{Eqz}) we find
\begin{eqnarray}\label{Eqzz}
\|\nabla v(t)\|_{L^{\infty}} &\lesssim & \big(\|\omega_{0}\|_{L^{a}\cap L^{\infty}} +t\|\nabla \theta_0\|_{L^{a} \cap L^{\infty}}e^{CV(t)}\big)\nonumber\\
&\times & \big(M_0+t+t\|\nabla \theta_0\|_{L^{a} \cap L^{\infty}}e^{CV(t)}+V(t)\big).
\end{eqnarray}
We choose $T>0$ such that $T$ satisfying
$$T\|\nabla \theta_0\|_{L^{a}\cap L^{\infty}}e^{CV(T)} \leq \min(1,\|\omega_{0}\|_{L^{a}\cap L^{\infty}}).$$
From (\ref{Eqzz}) we obtain
$$\|\nabla v(t)\|_{L^{\infty}} \lesssim \|\omega_{0}\|_{L^{a}\cap L^{\infty}} \big(M_0 +t+\int_0^t \|\nabla v(\tau)\|_{L^{\infty}}d\tau.\big), \forall t\in [0,T].$$
Gronwall's inequality allows us to deduce that
$$\|\nabla v(t)\|_{L^{\infty}} \lesssim  \|\omega_{0}\|_{L^{a}\cap L^{\infty}}(M_0 +t)e^{C\|\omega_{0}\|_{L^{a}\cap L^{\infty}}t}, \forall t\in [0,T].$$
We choose $T$ by the following formula
\begin{equation}\label{T}
T \triangleq \frac{1}{C\|\omega_0\|_{L^{a}\cap L^{\infty}}}\Log \bigg( 1 + \frac{\|\omega_0\|_{L^a\cap L^{\infty}}}{M_0\|\omega_0\|_{L^a\cap L^{\infty}}+1}\Log \big(1+\frac{C\min(\|\omega_0\|_{L^a\cap L^{\infty}},\|\omega_0\|^2_{L^a\cap L^{\infty}})}{\|\nabla \theta_0\|_{L^{\infty}}}\big)\bigg).
\end{equation}

Finally we get for all  $ t \leq T_{0}$

$$\|\nabla v(t)\|_{L^{\infty}} \leq M_0. $$
\end{proof}
\subsection{Existence and uniqueness} For the existence issue of the system \eqref{Eq1} we mollifier the initial data in a way that $v_{0,n}=S_{n}v_{0}, \theta_{0,n}=S_{n}\theta_0,$ where $S_n$ denotes the cut-off operator, see subsection \ref{Little-P}. Consider the approximation system 
\begin{equation}\label{Eq1n}
\left\{ \begin{array}{ll}
\partial_{t}v_{n}+v_{n}\cdot\nabla v_{n}+\nabla p_{n}=G_1(\theta_{n})\vec e_1+G_2(\theta_n)\vec e_2, & \\
\partial_{t}\theta_{n}+v_{n}\cdot\nabla \theta_{n}=0, & \\
\Div v_{n}=0, &\\
(v_n,\theta_n)_{| t=0}=(v_{0,n}, \theta_{0,n}).
\end{array} \right.\tag{NBn}
\end{equation} 
It is clear that $v_{0,n},\theta_{0,n}\in C^{\eta}$ with $\eta>1$, so the fact $G\in C^3$ implies in view of Theorem \ref{thm2.8} that 
\begin{eqnarray*}
\|G_i(\theta_{0,n})\|_{\eta}&\le &C\sup_{|y|\le C\|\theta_{0,n}\|_{L^\infty}}\|G_i^{[\eta]+2}(y)\|_{L^\infty}\|\theta_{0,n}\|_{\eta}\\
&\le&C\sup_{|y|\le C\|\theta_{0}\|_{L^\infty}}\|G_i^{[\eta]+2}(y)\|_{L^\infty}\|\theta_{0,n}\|_{\eta}.
\end{eqnarray*}
We have used in the last line that $\|\theta_{0,n}\|_{L^\infty}\le \|\theta_0\|_{L^\infty}$. Since $\eta>1$ is an arbiratry, we choose it in a way that $[\eta]=1$, we may deduce that 
\begin{equation*}
\|G_i(\theta_{0,n})\|_{\eta}\le C\sup_{|y|\le C\|\theta_{0}\|_{L^\infty}}\|G_i^{3}(y)\|\|\theta_{0,n}\|_{\eta}<\infty.
\end{equation*}
Meaning that $G_i(\theta_{0,n})$ also belongs to $C^\eta$, with $\eta>1$, it follows thanks to \cite{Chae-Kim-Nam-2} that \eqref{Eq1n} admits a unique local solution, that is $v_n,\theta_n,G_i(\theta_n)\in C\big([0,T^{\star}_n[;C^\eta\big)$, with the maximal life span $T^{\star}_n$ fullfils the following blow-up criterion
\begin{equation}\label{bu}
T^{\star}_{n} < \infty \Rightarrow \int_0^{T^{\star}_{n}} \|\nabla v_{n} (\tau)\|_{L^{\infty}} d\tau = +\infty.
\end{equation}
Besides, to uniform the quantities often mollified, we explore the properties of mollifier sequence, in particular the continuity of $S_n$ from $L^p$ into itself for $p\in[1,\infty]$ and an intense para-differential calulsus combined with Bony's decomposition, see \cite{Chemin2} we may write
$$\|\partial_{X_{0,\lambda}} \omega_{0,n} \|_{\epsilon-1} \lesssim \|\partial_{X_{0,\lambda}} \omega_{0} \|_{\epsilon-1} + \check{\|} X_{0,\lambda}\|_\epsilon \|\omega_0\|_{L^{\infty}}$$
$$\|\partial_{X_{0,\lambda}} \theta_{0,n} \|_{\epsilon-1} \lesssim \|\partial_{X_{0,\lambda}} \theta_{0} \|_{\epsilon-1} + \check{\|} X_{0,\lambda}\|_\epsilon \|\nabla\theta_0\|_{L^{\infty}}.$$
To close our claim, assume for some $n\in\NN,\; T_n^{\star}\le T_0$ with $T_0$ is expressed by \eqref{T} which justify all the previous a priori estimates. So, in accordance with Proposition \ref{Lip-v}, we conclude the following
\begin{equation}\label{v}
\|\nabla v_n(t)\|_{L^{\infty}} \leq M_0.
\end{equation}
$$ \|\omega_{n}(t)\|_{L^{a} \cap L^{\infty}} + \|\nabla \theta_{n}(t)\|_{L^{a} \cap L^{\infty}} \leq M_0,\quad \|\omega_{n}(t)\|_{L^{a} \cap L^{\infty}} + \|\nabla G_i(\theta_{n})(t)\|_{L^{a} \cap L^{\infty}}\le M_0$$
because 
\begin{eqnarray*}
\|\nabla G_i(\theta_{0,n})\|_{L^{a} \cap L^{\infty}} &\leq & \|G_i(\theta_{0,n})\|_{L^\infty} \|\nabla \theta_{0,n}\|_{L^{a} \cap L^{\infty}}\\
&\lesssim &\|\nabla \theta_{0}\|_{L^{a} \cap L^{\infty}}.
\end{eqnarray*}
Furthermore, 
$$\|G(\theta_{n}(t))\|^{\epsilon +1}_{X_{t,n}} + \|\omega_{n}(t)\|^{\epsilon}_{X_{t,n}} + \sup_{\lambda\in\Lambda}\|\partial_{X_{0,\lambda}} \Psi_n (t)\|_{\epsilon} \leq M_0,$$
with $\Psi_n$ refers to the associated flow of $v_n$. Accordingly \eqref{v} contradicts the blow-up criterion \eqref{bu} and therefore $T^{\star}_{n} \geq T.$ Finally, we explore the classical compactness argument to establish that the solutions sequence $(v_n,\theta_n)$ of the system \eqref{Eq1n}  converges when $n$ goes to infinity towards $(v,\theta)$ solution of the system (\ref{Eq1}).

\hspace{0.5cm}To treat the uniqueness issue for \eqref{Eq1}. It will be proven in the following space $\mathcal{M}_{T_0}$, with $\mathcal{M}_{T_0}=L^{\infty}([0,T_{0}],L^{p'} )\cap W^{1,\infty}$ for some $ 2<p'<\infty.$ Let us denote that this space is larger than the existence one because \eqref{Eq1} is of the hyperbolic type. Let $(v_i,\nabla p_i,\theta_i) \in \mathcal{M}_{T_0}, 1\leq i \leq 2$ be two solutions of the system (\ref{Eq1}), we set $\delta v = v_1 - v_2, \delta\theta = \theta_1 - \theta_2$ and  $\delta p = p_1 - p_2$ where the triplet $(\delta v, \delta \theta, \delta p)$ satisfy the following system
\begin{equation}\label{Diff-0} 
\left\{ \begin{array}{ll}
\partial_{t} \delta v+v_2\cdot\nabla \delta v = G(\theta_1)-G(\theta_2)-\nabla \delta p - \delta v\cdot\nabla v_1,& \\
 \partial_{t} \delta\theta+v_2\cdot\nabla\delta\theta= -\delta v\cdot\nabla\theta, & \\
\Div v=0, &\\
({v},{\theta})_{| t=0}=({v}_0,{\theta}_0).
\end{array} \right.
\end{equation}
The classical $L^{q}-$estimate for $\delta v$ provides
\begin{eqnarray}\label{delta-v-}
\\
\nonumber\|\delta v(t)\|_{L^{q}} \leq \|\delta v_0\|_{L^{q}} + \int_0^t(\|\delta v(\tau)\|_{L^{q}}\|\nabla v_1(\tau)\|_{L^\infty}+\|\delta p(\tau)\|_{L^{q}})d\tau + \int_0^t \|(G(\theta_1)-G(\theta_2))(\tau)\|_{L^{q}}d\tau.
\end{eqnarray}
For the term $\|(G(\theta_1)-G(\theta_2))(\tau)\|_{L^{q}},$ using Taylor's formula at order $1$ we get
$$ G_i(\theta_1)-G_i(\theta_2)=(\theta_1-\theta_2)\int_0^1 G'_i(\theta_2 +\sigma(\theta_1-\theta_2))dr.$$
Therefore
$$\|(G(\theta_1)-G(\theta_2))\|_{L^{q}} \leq \|\delta \theta\|_{L^{q}}\int_0^1 \|G'(\theta_2 +\sigma(\theta_1-\theta_2))\|_{L^{\infty}}dr.$$\\
On the other hand
$$\|G'(\theta_2 +\sigma(\theta_1-\theta_2))\|_{L^{\infty}} \leq \sup_{|x|\leq \|\theta_{2}+\sigma(\theta_1-\theta_2)\|_{L^{\infty}}}|G'_{i}(x)|.$$
Since $ 0< \sigma <1$ and $(\theta_1,\theta_2)$ satisfy two systems with the same initial data, i.e. $\|\theta_1(t)\|_{L^{\infty}} \leq \|\theta_{0}\|_{L^{\infty}}$ and $\|\theta_2(t)\|_{L^{\infty}} \leq \|\theta_{0}\|_{L^{\infty}}$ we further get
$$\|G'(\theta_2 +\sigma(\theta_1 - \theta_2))\|_{L^{\infty}} \leq \sup_{|x| \lesssim \|\theta_0\|_{L^{\infty}}}|G'_{i}(x)| \leq C .$$
Hence
\begin{equation}\label{EqG}
\int_0^t \|(G(\theta_1)-G(\theta_2)(\tau)\|_{L^{q}}d\tau \lesssim \int_0^t \|\delta \theta\|_{L^{q}}d\tau.
\end{equation}
In accordance with \eqref{delta-v-}, one has
\begin{equation}\label{Eq2.16}
\|\delta v(t)\|_{L^{q}} \lesssim \|\delta v_0\|_{L^{q}} + \int_0^t(\|\delta v(\tau)\|_{L^{q}}\|\nabla v_1(\tau)\|_{L^\infty}+\|\delta p(\tau)\|_{L^{q}})d\tau + \int_0^t \|\delta \theta (\tau)\|_{L^{q}}d\tau.
\end{equation}
Also $L^{q}-$estimate for the density allows us to write
\begin{equation}\label{Eq2.17}
\|\delta \theta (t)\|_{L^{q}} \leq \|\delta \theta_0\|_{L^{q}} + \int_0^t \|\delta v(\tau)\|_{L^{q}}\|\nabla \theta_1\|_{L^{\infty}} d\tau.
\end{equation}
Concerning the pressure term using the incompressibility condition and the identity $\Div(v_2\cdot\delta v) = \Div(v\cdot\nabla v_2)$ it holds that
\begin{eqnarray*}
\nabla \delta p &=& \nabla \Delta^{-1} \Div (-\delta v \cdot \nabla v_1 + G(\theta_1)-G(\theta_2)) -\nabla \Delta^{-1} \Div (v_2\cdot \nabla \delta v)\\
&=& \nabla \Delta^{-1} \Div (-\delta v \cdot \nabla(v_1+v_2)+G(\theta_1)-G(\theta_2)).
\end{eqnarray*}
Thus we obtain
\begin{eqnarray}
\|\nabla \delta p\|_{L^{q}} &\lesssim & \|\delta v\|_{L^{q}}(\|\nabla v_1\|_{L^{\infty}}+ \|\nabla v_2\|_{L^{\infty}})+\|G(\theta_1)- G(\theta_2)\|_{L^{q}}\nonumber\\  &\lesssim & \|\delta v\|_{L^{q}}(\|\nabla v_1\|_{L^{\infty}}+ \|\nabla v_2\|_{L^{\infty}})+\|\delta\theta\|_{L^{q}}.
\end{eqnarray}
Above, we have used the continuity of Riesz transform on $L^{q}$ with $1<q<\infty, $ putting \eqref{Eq2.17} into \eqref{Eq2.16}, it follows
$$\|\delta v(t)\|_{L^{q}} \lesssim \|\delta v_0\|_{L^{q}} + \int_0^t(\|\delta v(\tau)\|_{L^{q}}(\|\nabla v_1(\tau)\|_{L^\infty}+\|\nabla v_2(\tau)\|_{L^{\infty}})d\tau + \int_0^t \|\delta \theta (\tau)\|_{L^{q}}d\tau.$$
Inserting \eqref{Eq2.17} in the last inequality to obtain
$$\|\delta v(t)\|_{L^{q}} \lesssim \|\delta v_0\|_{L^{q}} + \|\delta \theta_0\|_{L^{q}} + \int_0^t(\|\delta v(\tau)\|_{L^{q}}(\|\nabla v_1(\tau)\|_{L^\infty}+\|\nabla v_2(\tau)\|_{L^{\infty}} + \|\nabla\theta_1\|_{L^{\infty}})d\tau.$$
Gronwall's lemma entails that
$$\|\delta v(t)\|_{L^{q}} \lesssim (\|\delta v_0\|_{L^{q}} + \|\delta \theta_0\|_{L^{q}})e^{C{\int_0^t(\|\nabla v_1(\tau)\|_{L^\infty}+\|\nabla v_2(\tau)\|_{L^{\infty}} + \|\nabla\theta_1\|_{L^{\infty}})d\tau}}.$$
We also obtain
$$\|\delta \theta(t)\|_{L^{q}} \lesssim \|\delta \theta_0\|_{L^{q}} + (\|\delta v_0\|_{L^{q}} + \|\delta \theta_0\|_{L^{q}})e^{ct}e^{C{\int_0^t(\|\nabla v_1(\tau)\|_{L^\infty}+\|\nabla v_2(\tau)\|_{L^{\infty}} + \|\nabla\theta_1\|_{L^{\infty}})d\tau}}.$$
So the proof is completed.

\subsection{Proof of Theorem \ref{thm1}}
As the boundary $\partial D_0$ is a Jordan curve of class $C^{1+\epsilon}$, meaning that  there exists in view of the definition \ref{Chart} a real function $f_0$ and a neighborhood $V_0$ such that $f_0 \in C^{1+\epsilon}, \; \nabla f_0(x) \neq 0$ on  $V_0$ and $\partial D_0 = f^{-1}_0(\lbrace0\rbrace)\cap V_0$. Let $\varphi$ be a smooth function satisfying
$$ \supp \varphi \subset V_0, \quad \varphi(x)=1,\quad \forall x\in V_1,$$
with $V_1$ is a small nighbrohood of $V_0.$ Now we construct an admissible family as follows:
$$X_{0,0}(x) = \nabla^{\perp}f_0(x)=\left( \begin{array}{ll}
-\partial_2 f_0(x) \\
\partial_1 f_0(x)
\end{array} \right), \quad X_{0,1}(x)=(1-\varphi(x))\left( \begin{array}{ll}
0 \\
1
\end{array} \right).$$
Set $X_0=(X_{0,\lambda})_{\lambda \in \lbrace0,1\rbrace},$ so we have $X_{0,\lambda}$ belongs to $C^\epsilon$ as well as its divergence, from Definition \ref{D1} we find that $X_0$ is an admissible family.
\hspace{0.5cm}By hypothesis $\theta_0 \in C^{1+\epsilon}$, this gives $\theta_0 \in C^{1+\epsilon}(X_{0,i})$ with $i\in \lbrace1,2\rbrace$.  On the other hand, we have $\partial_{\nabla^\bot f_0} \omega_0 =0.$ Since $1-\varphi\equiv 0$ on $V_1$ then $\partial_{X_{0,1}}\omega_0 = 0,$ or, in view of Theorem \ref{thm1}, the system \eqref{Eq1} has a unique local solution $v,\theta \in L^{\infty}([0,T],\Lip(\RR^2)).$
Next, we check the regularity of the transported initial domain $D_t$. We parametrize $\partial D_0$ by considering $x_0\in  \partial D_0, \gamma^0 \in C^\epsilon(\RR_+,\RR^2)$ satisfies the ordinary equation
$$\left\{ \begin{array}{ll}
\partial_\sigma \gamma^0(s)=X_{0,0}\big(\gamma^0(s)\big)\\
\gamma^0(0)=x_0,
\end{array} \right.$$
for every $t\geq 0,$
$$\gamma(t,s) \triangleq \Psi(t,\gamma^0(s)).$$
Differentating with respect to the variable $s$ we get
$$\partial_s \gamma (t,s) = (\partial_{X_{0,0}}\Psi)(t,\gamma^0(s)).$$
Thus we get $\text{ }\partial_{X_{0,0}}\Psi \in L^\infty([0,T_0],C^\epsilon)$ consequently $\gamma(t) \in L^\infty([0,T_0],C^{1+\epsilon})$. This completes the regularity persistence of the boundary $\partial D_t.$

\section{Singular vortex patches}\label{S4}

This section treats, especially the singular vortex patch stated in Theorem \ref{thm2} with more general initial data belonging to Yudovich class. To derive a local well-posedness topic for \eqref{Eq1} in the setting of singular patch we shall assume in addition that $G(\theta_0)$ is a constant environing the singularity in the direction $\vec e_1$ and $e_2$, where $\theta_0$ refers to initial density. Such assumption was imposed in the purpose to cancel the effects of singularity for the density process which is considered as a real drawback of this problem.

\hspace{0.5cm}The general version of Theorem \eqref{Eq1} will contain some necessary tools that were presented by Chemin in \cite{Chemin2} concerning the singular vortex patch for the Euler equations. First, let us denote that the formalism already accomplished for the first part differs from what we will do later because the initial boundary of the patch contains a finite set of singularities. This latter phenomenon contributes a loss of regularity for the family of vector fields in the sense of degenerating. To remedy this serious problem, we will truncate with an admissible family indexed by a truncate parameter.

\subsection{Singular patch tool box} We shall collect some technical tools which arise in the resolution of the singular patch problem. More precisely, we state the concept of family vector fields which respects the particularity of the singularities and detects some regularities of the velocity. To do so, we start with the following elementary definitions.
\begin{Definition}\label{D5}
Let $\Sigma$ be a closed set of $\RR^2$, for $h>0$ define $\Sigma_h$ the neighborhood of $\Sigma$ by
$$
\Sigma_h=\{x\in\RR^2 :d(x,\Sigma)\le h\}.
$$
We denote by $L(\Sigma)$ the set of the functions $v$ such that
$$ \|v\|_{L(\Sigma)} \triangleq \sup_{0<h\leq e^{-1}} \frac{\|v\|_{L^\infty(\Sigma^c_h)}}{-\Log h} < \infty,$$
where $\Sigma_h^{c}= \lbrace x\in \RR^2 : d(x,\Sigma)>h \rbrace.$

\end{Definition}
\hspace{0.5cm}Since the $\log$-Lipschitz class is a pivot tool in Yudovich's solutions which arises in the setting of a bounded and integrable vorticity. More precisely, we have.
\begin{Definition}\label{D6}
 The class of $\log$-Lipschitz functions denoted by $LL$ is defined by
$$ LL \triangleq \bigg\lbrace v\in L^\infty(\RR^2) : \|v\|_{LL} \triangleq \|v\|_{L^\infty} + \sup_{0<|x-y|<1} \frac{|v(x)-v(y)|}{|x-y|\Log\frac{e}{|x-y|}} < \infty \bigg\rbrace.$$
\end{Definition}
The relationship between the velocity and its vorticity in $LL-$space is given the following statement.
\begin{Lemma}\label{lem5}
For any $a\in]1,\infty[$ we have $\|v\|_{LL} \leq C \|\omega\|_{L^a\cap L^\infty},$ with $C=C(a)$ being a positive constant.
\end{Lemma}
As be explain in the introduction about the Yudovich's solutions, the velocity vector field belonging only to $LL$ and generates, according to Osgood lemma a unique flow $\Psi$ which is an homeomorph with respect to time and space variables and satisfying the classical ordinary equation 
\begin{equation*}
\left\{\begin{array}{ll}
\frac{\partial}{\partial t}\Psi(t,x)=v(t,\Psi(t,x))&\\
\Psi(0,x)=x.
\end{array}
\right.
\end{equation*}
The dynamic view of the transported sets and their complementary by the flow $\Psi$ associated to $v\in LL$ is given by the following.  
\begin{Lemma}\label{Lem6}
Let $A_0\subset\RR^2$ and $v\in L^{1}_{\loc}(\RR_+;$LL$)$.  We set $A(t) \triangleq \Psi(t,A_0)$, then we have
$$
\Psi\big(t,(A_0)^{c}_{h}\big) \subset \big(A(t)\big)^{c}_{\delta_{0,t}(h)},\quad\Psi\big(\tau,\Psi^{-1}\big(t,(A_t)^{c}_h\big)\big) \subset \big(A(\tau)\big)^{c}_{\delta_{\tau,t}(h)}\;\;  \forall \tau\in[0,t],
$$
with $$\quad \delta_{\tau,t}(h) \triangleq h^{\Exp\int_\tau^t \|v(\tau')\|_{LL}d\tau'}.$$
\end{Lemma}
Along this section we shall deal with different transport type equations, so it is legitimate to characterize its solution by some regularity. Especially we have.
\begin{Proposition}\label{prop4}
Let $(\epsilon,a)\in]-1,1[\times]1,\infty[$ and $v$ be a smooth divergence-free vector field. Set
$$W(t) \triangleq \big(\|\nabla v(t)\|_{L(\Sigma_t)} + \|\omega(t)\|_{L^a \cap L^\infty}\big)\Exp\bigg( \int_0^t \|v(\tau)\|_{LL}d\tau\bigg),\quad \Sigma_t = \Psi(t,\Sigma_0).$$
Let $f \in L^{\infty}_{\loc}([0,T],C^\epsilon)$ be a solution of the following inhomogeneous transport equation,
\begin{equation*}
\left\{ \begin{array}{ll}
\partial_{t}f+v\cdot\nabla f=g, \\
f_{| t=0}={f}_0,
\end{array} \right.
\end{equation*}\\
with $g=g_1+g_2$ is given and belongs to $L^1([0,T],C^\epsilon).$ We assume that $\supp f_0 \subset(\Sigma_0)^{c}_{h}$ and $\supp g(t) \subset(\Sigma_0)^{c}_{\delta(t,h)}$ for all $0\leq t \leq T$, and for some small $h$
$$\|g_2(t)\|_\epsilon \leq -C\Log(h)W(t)\|f(t)\|_\epsilon.$$
Then
$$\|f(t)\|_\epsilon \leq \|f_0\|_\epsilon h^{-C\int_0^t W(\tau)d\tau}+\int_0^t h^{-C\int_{\tau}^t W(\tau')d\tau'} \|g_1(\tau)\|_\epsilon d\tau,$$
where $C$ is a universal constant.
\end{Proposition}

\hspace{0.5cm}It is well-known that the striated regularity of the vorticity in the framework of singular patches requires a specific admissible family of vector fields indexed by a truncation parameter which is given in detail in \cite{Chemin2}.
\begin{Definition}\label{D7} Let $\epsilon \in ]0,1[$, $\Sigma$ be a closed subset of $\RR^2$ and $\Theta = (\alpha,\beta,\gamma)$ be a triplet of real numbers. A family $\mathcal{X}=(X_{\lambda,h})_{(\lambda,h)\in \Lambda\times ]0,e^{-1}]}$ of vector fields is said a $\Sigma$-admissible of order $\Theta$ if and only $X_{\lambda,h}, \Div X_{\lambda,h}\in C^{\epsilon}$ for every $(\lambda,h)\in \Lambda\times ]0,e^{-1}]$ and the following properties hold.
\begin{eqnarray*}
&&\supp(X_{\lambda,h}) \subset \Sigma^{c}_{h^{\alpha}},\; \forall (\lambda,h)\in \Lambda\times ]0,e^{-1}],\\
&&\inf_{h \in ]0,e^{-1}]} h^{\gamma}I(\Sigma_h,\mathcal{X}_h) > 0,\\
&&\sup_{h \in ]0,e^{-1}]} h^{-\beta}N_{\epsilon}(\Sigma_h,\mathcal{X}_h) < \infty,
\end{eqnarray*}
with the notation: for $\eta \geq h^{\alpha},$
$$I(\Sigma_\eta,\mathcal{X}_h) \triangleq \inf_{x\in \Sigma^{c}_{\eta}}\sup_{\lambda\in \Lambda}|X_{\lambda,h}(x)| \qquad \text{and}\qquad N_{\epsilon}(\Sigma_\eta,\mathcal{X}_h) \triangleq \sup_{\lambda \in \Lambda} \frac{\check{\|}X_{\lambda,h}\|_\epsilon}{I(\Sigma_\eta,\mathcal{X}_h)}$$
and
\begin{equation}\label{Eq32}
\|v\|^{\epsilon+k}_{\Sigma_\eta,\mathcal{X}_h} \triangleq N_{\epsilon}(\Sigma_\eta,\mathcal{X}_h)\sum_{|\alpha'|\leq k} \|\partial^{\alpha'}v\|_{L^\infty} + \sup_{\lambda\in\Lambda}\frac{\|\partial_{X_\lambda}v\|_{\epsilon+k-1}}{I(\Sigma_\eta,\mathcal{X}_h)}.
\end{equation}
\end{Definition}
\hspace{0.5cm}At this stage, we are ready to state a general statement of the Theorem \ref{thm2} which in turn covers not only the singular patches but expands to the solutions of Yudovitch's kind. To be precise, we will prove.
\begin{Theorem}\label{Thm2-gen}
Let $(\EE,h,a)\in ]0,1[\times]0,e^{-1}[\times]1,2[$ and $\Sigma_0$ be a closed subset of $\RR^2$. For $D_0$ a bounded domain of $\RR^2$ whose boundary $\partial D_0$ is a Jordan curve of $C^{1+\EE}$ regularity outside $\Sigma_0$ and an initial velocity $v_0$ in free-divergence, with $\omega_0\in L^a\cap L^\infty$. Assume that $\theta_0\in W_{G}^{1,a}(\RR^2)\cap W_{G}^{1,\infty}$ and for $i\in\lbrace1,2\rbrace, $  $G_i(\theta_0)$ is a constant $(\Sigma_0)_r$. Let $\mathcal{X}=(X_{\lambda,h})_{(\lambda,h)\in \Lambda\times ]0,e^{-1}]}$ be a $\Sigma$-admissible family of order $\Theta=(\alpha_0,\beta_0,\gamma_0)$ so that the following assertion holds.
\begin{equation*}
\sup_{h\in]1,e^{-1}]}h^{-\alpha_0}\|\theta_0\|^{\EE+1}_{(\Sigma_0)_h,(\mathscr{X}_0)_h}+\sup_{h\in]1,e^{-1}]}h^{-\beta_0}\|\omega_0\|^{\EE}_{(\Sigma_0)_h,(\mathscr{X}_0)_h}<\infty.
\end{equation*}
Then, there exists $T>0$ such that (\ref{Eq1}) has a unique local solution
\begin{equation*}
(\omega,\theta)\in L^\infty\big([0,T];L^a(\RR^2) \cap L^\infty(\RR^2)\big) \times L^\infty\big([0,T];W^{1,a}(\RR^2) \cap W^{1,\infty}(\RR^2)\big).
\end{equation*}
Besides, the velocity vector field $v$ is Lipschitz outside the set $\Sigma_t$, with $\Sigma_t = \Psi(t,\Sigma_0)$ in the following way
\begin{equation*}
\sup_{h\in(0,e^{-1}]}\frac{\|\nabla v(t)\|_{L^{\infty}((\Sigma_t)^{c}_{h})}}{-\log h}\in L^\infty([0,T]).
\end{equation*}
\end{Theorem}
The proof of general version Theorem \ref{Thm2-gen} will be treated in detail in multiple subsections, we will start by some a priori estimates remarkably important in our analysis.

\subsection{A priori estimates for the buoyancy term and density}\label{S5}
In this paragraph, we are concerned with a priori estimate of the $L^p$ type for the buoyancy force $G(\theta)$, the density, and the vorticity. Keeping in mind that the velocity lacks its regularity through the time, so $G(\theta_0)$ is constant nearby the singularities is an axle assumption in the sense to get back what lack of regularity. Without this assumption the problem becomes unsolvable. First, we embark with the following estimates.
\begin{Proposition}\label{prop5}
Let $(v,\theta)$ be a smooth solution of (\ref{Eq1}) defined on $[0,T].$ with $v$ in divergence free. Let $\Sigma_0$ be a closed set of $\RR^{2}$ and for $i\in\lbrace1,2\rbrace$ assume that $G_i(\theta_0)$ is a constant in the set $(\Sigma_0)_{r}= \lbrace x \in \RR^{2} : d(x,\Sigma_0) \leq r \rbrace$ for some $r \in ]0,e^{-1}[$. Then for all $p \in [1,+\infty]$ and for any $t \leq T$ we have

\begin{equation}\label{prop5:1}
\|\nabla \big(G_i(\theta(t))\big)\|_{L^{p}} \lesssim \|\nabla \theta_{0}\|_{L^{p}}r^{-C\int_0^t W(\tau)d\tau}
\end{equation}


and
\begin{equation}\label{prop5:2}
 \|\omega(t)\|_{L^{p}} \lesssim \|\omega_{0}\|_{L^{p}} + t\|\nabla \theta_{0}\|_{L^{p}}r^{-C\int_0^t W(\tau)d\tau}.
\end{equation}
\end{Proposition}
\begin{proof}

Let $i\in\lbrace1,2\rbrace,$ since $G_i(\theta)$ is a solution for the following transport equation,
\begin{equation}\label{Lun-0}
\left\{ \begin{array}{ll}
 \partial_{t}G_i(\theta) +v\cdot\nabla G_i(\theta)=0,\\
 G(\theta)_{|_{t=0}}=G(\theta_0).
\end{array} \right.
\end{equation}
So, by applying the partial derivative $\partial_j$, one has

\begin{equation}\label{Eq34}
\partial_t \partial_j \big(G_i(\theta)\big) + v\cdot \nabla (\partial_j \big(G_i(\theta)\big)) = -(\partial_j v \cdot \nabla) G_i(\theta).
\end{equation}
Or, an integration by parts and the condition $\Div v=0$ provide for $p \in [1,+\infty]$, 
\begin{eqnarray}\label{Eq35}
\|\partial_j \big(G_i(\theta(t))\big)\|_{L^{p}} &\leq & \|\partial_j      \big(G_i(\theta_0)\big)\|_{L^{p}} + \int_0^t \|\partial_j v\cdot\nabla\big(G_i(\theta(\tau))\big)\|_{L^p}d\tau\\ \nonumber
\nonumber&\leq & \|G'_i(\theta_0) \partial_j \theta_0\|_{L^p} + \int_0^t \|\partial_j v\cdot\nabla\big(G_i(\theta(\tau))\big)\|_{L^p}d\tau\\ \nonumber
\nonumber&\leq & \|G'_i(\theta_0)\|_{L^\infty} \|\partial_j \theta_0\|_{L^p} + \int_0^t \|\partial_j v\cdot\nabla\big(G_i(\theta(\tau))\big)\|_{L^p}d\tau\\
\nonumber&\lesssim & \|\partial_j \theta_0\|_{L^p} + \int_0^t \|\partial_j v\cdot\nabla\big(G_i(\theta(\tau))\big)\|_{L^p}d\tau.
\end{eqnarray}
In the last line we have used the fact $\|G'_i(\theta_0)\|_{L^\infty}$ is bounded, see \eqref{G'-bounded}. 

\hspace{0.5cm}To treat the term $\|\partial_j v\cdot\nabla\big(G_i(\theta(\tau))\big)\|_{L^p}$. On the one hand, we employ for $i\in\lbrace1,2\rbrace$ that $G_i\big(\theta(\tau,x)\big) =G_i\big(\theta_0(\Psi^{-1}(\tau,x))\big)$. On the other hand, $G(\theta_0)$ is constant over the set $ (\Sigma_0)_{r}$, so $G_i(\theta(\tau))$ it is also over the transported set $\Psi(\tau,(\Sigma_0)_{r}) $, one deduce that $\nabla\big(G_i(\theta(\tau))\big)= 0 $ in $\Psi(\tau,(\Sigma_0)_{r})$. We further get
$$ \supp \nabla G_i(\theta(\tau)) \subset \big(\Psi(\tau,(\Sigma_0)_{r})\big)^c = \Psi(\tau,(\Sigma_0)_{r}^c).$$
Owing to Lemma \ref{Lem6}, we infer that
$$\supp \nabla \big(G_i(\theta(\tau))\big) \subset  (\Sigma_t)^{c}_{\delta_{(\tau,r)}},\quad \delta_{(\tau,r)}= r^{\displaystyle{e^{\int_0^{\tau} \|v(\tau')\|_{LL}d\tau'}}}.$$
Therefore we find
$$\|\partial_j v\cdot\nabla\big(G_i(\theta(\tau))\big)\|_{L^p} \leq \|\nabla v(\tau)\|_{L^{\infty}\big((\Sigma_\tau)^{c}_{\delta_{\tau(r)}}\big)} \|\nabla\big(G_i(\theta(\tau))\big)\|_{L^p}.$$
By means of Definition \ref{D5}, it holds
\begin{eqnarray*}
\|\nabla v(\tau)\|_{L^{\infty}\big((\Sigma_\tau)^{c}_{\delta_{\tau(r)}}\big)} &\leq & -\|\nabla v(\tau)\|_{L^{\infty}(\Sigma_\tau)}\Log\delta_{\tau(r)}\\
&\leq & -(\Log r)\|\nabla v(\tau)\|_{L^{\infty}(\Sigma_\tau)}\displaystyle{e^{\int_0^{\tau} \|v(\tau')\|_{LL}d\tau'}}\\
&\leq & \displaystyle{-W(\tau)\Log(r)}.
\end{eqnarray*}
As a consequence
\begin{equation*}
\|\partial_j v\cdot\nabla\big(G_i(\theta(\tau))\big)\|_{L^p}\le \displaystyle{-W(\tau)\Log(r)}\|\nabla\big(G_i(\theta(\tau))\big)\|_{L^p}.
\end{equation*}
Plug the last estimate in \eqref{Eq35} to obtain
$$\|\nabla \big(G_i(\theta(t))\big)\|_{L^{p}} \lesssim \|\nabla \theta_{0}\|_{L^{p}}-C\Log r \int_0^t \|\nabla \big(G_i(\theta(\tau))\big)\|_{L^{p}}W(\tau)d\tau.$$
Via Gronwall's inequality we may write
\begin{equation}\label{o1}
\|\nabla \big(G_i(\theta(t))\big)\|_{L^{p}} \lesssim \|\nabla \theta_{0}\|_{L^{p}}r^{-C\int_0^t W(\tau)d\tau}.
\end{equation}
Let us move to estimate \eqref{prop5:2}. For $p\in[1,\infty]$, the classical $L^p-$estimate and maximum principle for $\omega-$ equation leads to
$$\|\omega(t)\|_{L^p} \leq \|\omega_{0}\|_{L^p}+ \int_0^t \sum_{i=1}^2\|\nabla \big(G_i(\theta(\tau))\big)\|_{L^{p}}.$$
In view of \eqref{o1} we write 

$$\|\omega(t)\|_{L^{p}} \lesssim \|\omega_{0}\|_{L^{p}} + t\|\nabla \theta_{0}\|_{L^{p}}r^{-C\int_0^t W(\tau)d\tau}.$$

Proposition \ref{prop5} is then proved.
\end{proof}
\subsection{A priori estimates for the admissible family and striated regularity of the voticity}The following properties are standard about the family of a vector field supported far from a closed set in the plane and arises in the striated regularity of the vorticity. Specifically, we have.
\begin{Proposition}\label{prop6}
Let $(\EE,a)\in]0,1[\times]1,\infty[$, $\Sigma_0$ be closed subset of $\RR^2$ and $X_0$ be a vector field of class such that $X_0, \Div X_0 \in C^\epsilon$ which satisfying $ \supp X_0 \subset (\Sigma_0)^c_h$. Let $(v,\theta)$ smooth solution of the system \ref{Eq1} defined on time interval $[0,T]$ with inital data $(v_0,\theta_0)$. Assume that for $i\in\lbrace 1,2 \rbrace, $ $G_i(\theta_0) $ is a constant in the set $(\Sigma_0)_r$ and $X_t$ satisfies the following system
\begin{equation}
\left\{ \begin{array}{ll}\label{Eq36}
 \partial_{t}X_t +v\cdot\nabla X_t= \partial_{X_t}v,\\
 X_{|_{t=0}}=X_0.
\end{array} \right.
\end{equation}
Then
\begin{equation}\label{prop6:1}
\supp X_t \subset (\Sigma_t)^c_{\delta_t(h)}.
\end{equation}
\begin{equation}\label{Eqdiv}
\|\Div X_t\|_\epsilon \leq \|\Div X_0\|_\epsilon h^{-C\int_0^t W(\tau)d\tau},
\end{equation}
\begin{eqnarray}\label{prop6:2}
\nonumber\check{\|}X_{t}\|_{\epsilon} + \|\partial_{X_{t,\lambda}}\omega\|_{\epsilon-1} &\leq & C\big(\check{\|}X_{0,\lambda}\|_{\epsilon}+ \|\partial_{X_{0,\lambda}}\omega_{0}\|_{\epsilon-1}+\|\theta_0\|_{\epsilon}\|\partial_{X_{0,\lambda}}\theta_0\|_{\epsilon}\big) e^{Ct} e^{-C\int_0^t W(\tau)d\tau}  \\
&\times&e^{t\|\nabla \theta_{0}\|_{L^{\infty}} r^{-C\int_0^t W(\tau)d\tau}}.
\end{eqnarray}
\begin{proof}
For \eqref{prop6:1}, using the fact $\supp X_0\subset (\Sigma_0)_{h}^{c}$, but $X_t(x)=X_{0}(\Psi(t, \Psi^{-1}(t,x)))$ allows us to write $\supp X_t\subset\Psi(t,(\Sigma_0)_{h}^c)$. Then thanks to Lemma \ref{Lem6}, we deduce that $\supp X_t\subset(\Sigma_t)^c_{\delta_t(h)}$.

\hspace{0.5cm}To estimate \eqref{Eqdiv} we must apply the operator "$\Div$" to \eqref{Eq36}, so that $\Div v=0$ yields
$$(\partial_t + v\cdot\nabla)\Div X_t = 0.$$
In view Proposition \ref{prop4} we find that
$$\|\Div X_t\|_\epsilon \leq \|\Div X_0\|_\epsilon h^{-C\int_0^t W(\tau)d\tau}.$$
\hspace{0.5cm}Come back to \eqref{prop6:2}. 
The fact that $\partial_{X_t} v(t) = g_1(t) + g_2(t)$, so Biot-Savart law combined, an intense paradifferential calculus and the definition of $\delta_{t}(h)$, see \cite{Chemin2} allow us to write
$$\|g_1(t)\|_{\epsilon} \leq C\|\partial_{X_t}\omega(t)\|_{\epsilon-1} + C\|\Div X_t\|_\epsilon\|\omega(t)\|_{L^\infty}$$
and
$$\|g_2(t)\|_{\epsilon} \leq -C\|X_t\|_\epsilon W(t)\Log h.$$
Gathering the last tow estimates, then a new use of Proposition \ref{prop4} for \eqref{Eq36} combined with \eqref{Eqdiv} we immediately obtain Due to its ubiquity in
\begin{eqnarray*}
\|X_t\|_\epsilon &\lesssim& \|X_0\|_\epsilon h^{-C\int_0^t W(\tau)d\tau} + \|\Div X_0\|_\epsilon\int_0^t \|\omega(\tau)\|_{L^{\infty}} h^{-C\int_\tau^t W(\tau')d\tau'}d\tau\\
&&+ \int_0^t \|\partial_{X_\tau}\omega(\tau)\|_{\epsilon-1} h^{-C\int_\tau^t W(\tau')d\tau'}d\tau.
\end{eqnarray*}
Exploring the definition of $W$ to write $\int_0^t \|\omega(\tau)\|_{L^\infty} d\tau \leq h^{-C \int_0^t W(\tau)d\tau}$, the previous estimate becomes
\begin{eqnarray}\label{Lun-1}
\|X_t\|_\epsilon  &\lesssim& \|X_0\|_\epsilon h^{-C\int_0^t W(\tau)d\tau} + \|\Div X_0\|_\epsilon h^{-C\int_0^t W(\tau)d\tau}\int_0^t\|\omega(\tau)\|_{L^{\infty}}d\tau\\
\nonumber&&+ \int_0^t \|\partial_{X_\tau}\omega(\tau)\|_{\epsilon-1} h^{-C\int_\tau^t W(\tau')d\tau'}d\tau.
\end{eqnarray}

\hspace{0.5cm}Finally, we treat carefully the term $\|\partial_{X_{t,\lambda}}\omega\|_{\epsilon-1}$. For this purpose, we apply the directional derivative $\partial_{X_{t,\lambda}}$ to $\omega-$equation one may write again
$$ (\partial_{t}+v\cdot\nabla)\partial_{X_{t,\lambda}}\omega = \partial_{X_{t,\lambda}}(\partial_{1}(G_{2}(\theta))-\partial_{2}(G_{1}(\theta))),$$
again Proposition \ref{prop4} leads
\begin{equation}\label{Lun-2}
\nonumber\|\partial_{X_t}\omega(t)\|_{\epsilon-1} \lesssim \|\partial_{X_0}\omega_0\|_{\epsilon-1} h^{-C \int_0^t W(\tau)d\tau}+\sum_{i=1}^{2}\int_0^t\|\partial_{X_{t,\lambda}}(\partial_{1}(G_{i}(\theta))\|_{\epsilon-1}h^{-C \int_0^t W(\tau)d\tau}d\tau.
\end{equation}
For the term $\|\partial_{X_{t,\lambda}}(\partial_{1}(G_{i}(\theta))\|_{\epsilon-1}$, let $i\in\{1,2\}$, we use the fact that
$$\partial_{X_{t,\lambda}}(\partial_{1}(G_{i}(\theta))= \partial_1\big(\partial_{X_{t,\lambda}}G_i(\theta)\big)+\big[\partial_{X_{t,\lambda}},\partial_1\big]G_i(\theta).$$
Therefore
\begin{equation*}
\|\partial_{X_{t,\lambda}}(\partial_{1}(G_{i}(\theta))\|_{\epsilon-1}  \leq \|\partial_{1}(\partial_{X_{t,\lambda}}(G_{i}(\theta))\|_{\epsilon-1} + \|\big[\partial_{X_{t,\lambda}},\partial_1\big]G_i(\theta)\|_{\epsilon-1}.
\end{equation*}
On the other hand from \eqref{com:1} we have $\big[\partial_X,\partial_1\big]G_i(\theta)= -(\partial_1 X_{t,\lambda})\cdot\nabla G_i(\theta)$. Accordingly finds out
$$\|\partial_{X_{t,\lambda}}(\partial_{1}(G_{i}(\theta))\|_{\epsilon-1} \leq \|\partial_{X_{t,\lambda}}(G_{i}(\theta))\|_{\epsilon} + \|\partial_{1}(X_{t,\lambda})\cdot\nabla(G_{2}(\theta))\|_{\epsilon-1}.$$
Since $G_i\circ\theta \in L^\infty$, the Corrolary \ref{Cor1} provides

\begin{equation*}\label{Eq23}
\|\partial_{X_{t,\lambda}}(\partial_{1}(G_{i}(\theta))\|_{\epsilon-1} \lesssim
\|\partial_{X_{t,\lambda}}(G_{i}(\theta))\|_{\epsilon} + \|\nabla\big(G_i(\theta(\tau)\big)\|_{L^{\infty}}\check{\|}X_{t,\lambda}\|_{\epsilon},\quad 1\le i\le 2.
\end{equation*}
Insert the last estimate in \eqref{Lun-1}, we end up with
\begin{eqnarray}\label{Lun-3}
\nonumber\|\partial_{X_{t,\lambda}}\omega(t)\|_{\epsilon-1} &\lesssim & \|\partial_{X_0}\omega_{0}\|_{\epsilon-1} h^{-C \int_0^t W(\tau)d\tau} + \sum_{i=1}^{2}\int_0^t \|\nabla (G_i(\theta)) \|_{L^{\infty}}\check{\|}X_{\tau,\lambda}\|_{\epsilon}h^{-C \int_\tau^t W(\tau')d\tau'}d\tau\\
&+& \sum_{i=1}^{2}\int_0^t \|\partial_{X_{\tau,\lambda}}G_i(\theta)\big)\|_{\epsilon} h^{-C \int_\tau^t W(\tau')d\tau'}d\tau.
\end{eqnarray}
The fact that $G_i(\theta_0)$ is constant in the set $(\Sigma_0)_r$, then $\partial_{X_{0,\lambda}}G_i(\theta_0)\equiv 0$ in the set $(\Sigma_0)_{r}$ this gives
\begin{equation}\label{Lun-30} 
\supp \big(\partial_{X_{0,\lambda}}(G_i(\theta_0))\big) \subset (\Sigma_0)^{c}_r.
\end{equation}
On the other hand, $G(\theta)$ satisfies a transport equation of the type \eqref{Lun-0}, so with the aid that $\partial_{X_{t,\lambda}}$ commutes with the transport operator $\partial_t+v\cdot\nabla$, that is
\begin{equation}\label{Lun-300}
\left\{ \begin{array}{ll}
 \partial_{t}\partial_{X_{t,\lambda}}G(\theta) +v\cdot\nabla \partial_{X_{t,\lambda}}G(\theta)=0,\\
 \partial_{X_{t,\lambda}}G(\theta)_{|_{t=0}}=\partial_{X_{0,\lambda}}G(\theta_0).
\end{array} \right.
\end{equation}
In account \eqref{Lun-300}, we note that Proposition \ref{prop4} for \eqref{Lun-30} affords us to write 
$$\|\partial_{X_{t,\lambda}}\big(G_{i}(\theta)\big)\|_{\epsilon}\leq \|\partial_{X_{0,\lambda}}\big(G_{i}(\theta_0)\big)\|_{\epsilon}h^{-C \int_0^t W(\tau)d\tau}.$$
Thus we get from \eqref{Eqalg}
$$\|\partial_{X_{t,\lambda}}\big(G_{i}(\theta)\big)\|_{\epsilon} \lesssim \|\theta_0\|_{\epsilon}\|\partial_{X_{0,\lambda}}\theta_0\|_{\epsilon} h^{-C \int_0^t W(\tau)d\tau}.$$
Consequently, \eqref{Lun-3} takes the form
\begin{eqnarray}\label{Lun-00}
\|\partial_{X_{t,\lambda}}\omega\|_{\epsilon-1} &\lesssim & h^{-C \int_0^t W(\tau)d\tau}(\|\partial_{X_{0,\lambda}}\omega_{0}\|_{\epsilon-1} + \|\theta_0\|_{\epsilon}\|\partial_{X_{0,\lambda}}\theta_0\|_{\epsilon}t)\\
 \nonumber&+&  \sum_{i=1}^{2}\int_0^t\|\nabla (G_i(\theta)) \|_{L^{\infty}}\check{\|}X_{t,\lambda}\|_{\epsilon}h^{-C \int_{\tau'}^\tau W(\tau')d\tau'}d\tau).
\end{eqnarray}
For the sake of brevity, set $$\mathfrak{C}(t) \triangleq \big(\|\partial_{X_{t,\lambda}}\omega\|_{\epsilon-1}+\check{\|}X_{t,\lambda}\|_{\epsilon}\big)h^{C \int_0^t W(\tau)d\tau}.$$
Hence \eqref{Lun-1} and \eqref{Lun-00} permit us to write 

$$ \mathfrak{C}(t) \lesssim \mathfrak{C}(0) + \|\theta_0\|_{\epsilon}\|\partial_{X_{0,\lambda}}\theta_0\|_{\epsilon}t +\int_0^t \bigg(\sum_{i=1}^{2}\|\nabla (G_i(\theta)) \|_{L^{\infty}}+1 \bigg)\Phi(\tau)d\tau.$$
By Gronwall's inequality we get
$$ \mathfrak{C}(t) \lesssim \big(\mathfrak{C}(0) + \|\theta_0\|_{\epsilon}\|\partial_{X_{0,\lambda}}\theta_0\|_{\epsilon}\big)\Exp\bigg(C  \sum_{i=1}^{2}\int_0^t\|\nabla (G_i(\theta)) \|_{L^{\infty}} d\tau +Ct\bigg).$$
According to Proposition \ref{prop5}, we conclude that
\begin{eqnarray*}
\nonumber\check{\|}X_{t}\|_{\epsilon} + \|\partial_{X_{t,\lambda}}\omega\|_{\epsilon-1} &\leq & C\big(\check{\|}X_{0,\lambda}\|_{\epsilon}+ \|\partial_{X_{0,\lambda}}\omega_{0}\|_{\epsilon-1}+\|\theta_0\|_{\epsilon}\|\partial_{X_{0,\lambda}}\theta_0\|_{\epsilon}\big) e^{Ct} e^{-C\int_0^t W(\tau)d\tau}  \\
&\times&e^{t\|\nabla \theta_{0}\|_{L^{\infty}} r^{-C\int_0^t W(\tau)d\tau}},
\end{eqnarray*}
and this ends the proof.

\end{proof}
\end{Proposition}
\subsection{Lipschitz norm of the velocity outside the singular set}
\hspace{0.5cm}This section concerns by bounding the Lipschitz norm of the velocity outside the singular set $\Psi(t,\Sigma_0)$ by exploring the striated regularity of its vorticity. This is regarded as the centerpiece in the proof of Theorem \ref{Thm2-gen}. 
\begin{Proposition}\label{prop7}
Given $(\EE,r,a)\in]0,1[\times]0,e^{-1}[\times]1,\infty[$ and $\Sigma_0$ be a closed set of $\RR^2$. Let $\mathcal{X}_0=\big(X_{0,\lambda,h}\big)_{(\lambda,h) \in \Lambda\times]0, e^{-1}[}$ be a $\Sigma_0$-admissible family of order $\Theta_0 = (\alpha,\beta_0,\gamma_0)$ and $(v,\theta)$ be a smooth solution of the system \eqref{Eq1} defined on a time interval $[0,T^{\star}[$. Assume that $\omega_0 , \nabla \theta_0 \in L^{a}$, $G_i(\theta_0), i\in\lbrace1,2\rbrace$ is constant over $(\Sigma_0)_r$, and
\begin{equation*}
\sup_{0<h\leq e^{-1}} h^{-\beta_0}\|\theta_0\|^{\epsilon+1}_{(\Sigma_0)_h,(\mathcal{X}_{0})_h} +  \sup_{0<h\leq e^{-1}} h^{-\beta_0}\|\omega_0\|^{\epsilon}_{(\Sigma_0)_h,(\mathcal{X}_{0})_h} < \infty.
\end{equation*}
Then there exists $T$ , $0<T<T^{\star}$ such that
$$ \|\nabla v(t)\|_{L(\Sigma(t))} \in L^{\infty}\big([0,T]\big).$$
\end{Proposition}
\begin{proof} For $\mathscr{X}_{0}=(X_{0,\lambda,h})_{\lambda\in\Lambda}$ an initial family of vector field. The dynamic of such family by the flow $\Psi$ through the time is a time-dependent family $\mathscr{X}_{t}=(X_{t,\lambda,h})_{\lambda\in\Lambda}$ defined by the well-known {\it Pushforward} process, that is
\begin{equation*}
X_{t,\lambda,h}(x)\triangleq X_{0,\lambda,h}\big(\Psi(t,\Psi^{-1}(t,x))\big).
\end{equation*}
To make the presentation more convenient, set $Z_{t,\lambda,h} \triangleq X_{0,\lambda,h}( \Psi(t,x))$. It is clear to verify that $Z_{t,\lambda,h}$ solves the following equation 
$$\partial_t Z_{t,\lambda,h}(x) = \nabla v(t,\Psi(t,x)) \cdot Z_{t,\lambda,h}(x).$$
On the other hande, for $t>0$ fixed and for all $\tau \in [0,T]$, put $ Y_{\tau,\lambda,h}(x) \triangleq Z_{t-\tau,\lambda,h}(x)$. We observe that $Y_{\tau,\lambda,h}(x)$ is a solution of the equation
$$\partial_\tau Y_{\tau,\lambda,h}(x) = - \nabla v(t-\tau,\Psi(t-\tau,x)) \cdot Y_{t-\tau,\lambda,h}(x),$$
so, Gronwall inequality provides
$$|Y_{t,\lambda,h}(x)| \leq |Y_{0,\lambda,h}(x)|\Exp\Big(\int_0^t|\nabla v(\tau,\Psi(\tau,x))|d\tau\Big).$$
Equivalently, we have
 $$|Y_{0,\lambda,h}(x)| \leq |Y_{t,\lambda,h}(x)|\Exp\Big(\int_0^t|\nabla v(\tau,\Psi(\tau,x))|d\tau\Big),$$
which implies
 $$|X_{0,\lambda,h}(x)| \leq |X_{t,\lambda,h}(x)|\Exp\Big(\int_0^t|\nabla v(\tau,\Psi(\tau,x))|d\tau\Big).$$
Hence, denoting $\delta_{t}^{-1}\triangleq h^{\Exp(\int_0^t \|v(\tau)\|_{LL}d\tau)}$ the inverse of $\delta_t$, it happens
\begin{eqnarray}\label{eq39}
\inf_{x\in(\Sigma_t)^{c}_{\delta_{t}^{-1}(h)}} \sup_{\lambda\in\Lambda} |X_{0,\lambda,h}\big(\Psi^{-1}(t,x)\big)| &\leq& \inf_{x\in(\Sigma)^c_{\delta_{t}^{-1}(h)}} \sup_{\lambda\in\Lambda}|X_{t,\lambda,h}(x)|  \\ \nonumber
&\times& \Exp\bigg(\int_0^t \|\nabla v \big(\tau,\Psi(\tau,\Psi^{-1}(\tau,\cdot)\big)\|_{L^\infty((\Sigma_t)^c_{\delta_{t}^{-1}(h)})}d\tau\bigg),
\end{eqnarray}
Since $\delta_t\big(\delta_{t}^{-1}(h)\big) = h$ and from Lemma \ref{Lem6} we can show that
\begin{equation}\label{(i)}
\Psi^{-1}\bigg(t,\big(\Sigma_t\big)^c_{\delta_{t}^{-1}(h)}\bigg) \subset \big(\Sigma_0\big)^c_{\delta_t(\delta_{t}^{-1}(h))} = \big(\Sigma_0\big)^c_h.
\end{equation}
\begin{equation}\label{(ii)}
\Psi\bigg(\tau,\Psi^{-1}\big(t,\big(\Sigma_t\big)^c_{\delta_t ^{-1}(h)}\big)\bigg) \subset (\Sigma_\tau\big)^c_{\delta_\tau(\delta_{t} ^{-1}(h))} \subset (\Sigma_\tau)^c_h.
\end{equation}
We treat the both sides of \eqref{eq39} separately. First, for the l.h.s, we make use to \eqref{(i)} and Definition \ref{D7}, it follows
\begin{eqnarray}\label{eq40}
\inf_{x\in(\Sigma_t)^{c}_{\delta_{t}^{-1}(h)}} \sup_{\lambda\in\Lambda} |X_{0,\lambda,h}\big(\Psi^{-1}(t,x)\big)| &=& \inf_{y\in\Psi^{-1}\big(t,(\Sigma_t)^{c}_{\delta_{t}^{-1}(h)}\big)} \sup_{\lambda\in\Lambda} |X_{0,\lambda,h}\big(y)\big)|\\ \nonumber
&\geq & \inf_{y\in(\Sigma_0)^{c}_h} \sup_{\lambda\in\Lambda} |X_{0,\lambda,h}\big(y)|\\ \nonumber
&\geq & I\big((\Sigma_0)_h , (\mathcal{X}_0)_h\big).
\end{eqnarray}
Second, for the r.h.s. also, in view of \eqref{(ii)} we write
\begin{eqnarray*}
\|\nabla v\big( \tau , \Psi(\tau,\Psi^{-1}(t,.))\big)\|_{L^{\infty}\big((\Sigma_t)^c_{\delta ^{-1}_t(h)}\big)}  &\leq &  \|\nabla v(\tau)\|_{L^{\infty}((\Sigma_\tau)^c_h)}\\
&\leq & -(\log h)\|\nabla v(\tau)\|_{L(\Sigma_\tau)}\\
&\leq & -(\log h)W(\tau).
\end{eqnarray*}
The preceding estimate combined with \eqref{eq39} and \eqref{eq40} we immediately obtain
\begin{equation}\label{eq41}
 I\big((\Sigma_t)_{\delta^{-1}_t(h)}, (\mathcal{X}(t))_h\big)h^{-\int_0^t W(\tau)d\tau} \geq I\big((\Sigma_0)_h, (\mathcal{X}(0))_h\big).
\end{equation}
At this stage, taking $\mathfrak{D}(t) \triangleq \|\partial_{X_{t,\lambda,h}}\omega\|_{\epsilon-1}+\|\omega(t)\|_{L^\infty}\check{\|}X_{t,\lambda}\|_{\epsilon}.$ Thus Proposition \ref{prop6} and Proposition \ref{prop5} lead to

\begin{eqnarray*}
\mathfrak{D}(t) & \lesssim & e^{Ct}\big(1+\|\omega_0\|_{L^\infty})(\check{\|}X_{0,\lambda,h}\|_\epsilon + \|\partial_{X_{0,\lambda,h}} \omega_0\|_{\epsilon-1} + \|\theta_0\|_{\epsilon} \|\partial_{X_{0,\lambda,h}}\theta_0\|_{\epsilon})\\
&\times &e^{Ct\|\nabla\theta_0\|_{L^\infty}r^{-C\int_0^t W(\tau)d\tau}} h^{-C\alpha_0\int_0^t W(\tau)d\tau}.
\end{eqnarray*}
We link the estimate \eqref{eq41} with Definition \ref{D7}, thus a straigtforward compution allows us to write
\begin{eqnarray*}
\mathfrak{D}(t) &\leq & e^{Ct}\bigg( N_{\epsilon}\big((\Sigma_0)_h , (\mathcal{X}_0)_h\big) + ( 1 + \|\omega_0\|_{L^\infty})\Big(\|\omega_0\|^{\epsilon}_{(\Sigma_0)_h , (\mathcal{X}_0)_h} + \|\theta_0\|_{\epsilon}\|\theta_0\|^{\epsilon+1}_{ (\Sigma_0)_h , (\mathcal{X}_0)_h }\Big)\bigg)\\
&\times & I{\big((\Sigma_0)_h , (\mathcal{X}_0)_h}\big) h^{-C\alpha_0\int_0^t W(\tau)d\tau}\times \Exp(Ct\|\nabla\theta_0\|_{L^\infty}r^{-C\int_0^t W(\tau)d\tau})\\
&\leq & Ce^{Ct} \sup_{0<h\leq e^{-1}} h^{-\beta_0}\bigg( N_{\epsilon}\big((\Sigma_0)_h , (\mathcal{X}_0)_h\big) + ( 1 + \|\omega_0\|_{L^\infty})\Big(\|\omega_0\|^{\epsilon}_{(\Sigma_0)_h , (\mathcal{X}_0)_h} + \|\theta_0\|_{\epsilon}\|\theta_0\|^{\epsilon+1}_{ (\Sigma_0)_h , (\mathcal{X}_0)_h }\Big)\bigg)\\
&\times & I{\big((\Sigma_t)_{\delta^{-1}_t(h)}} , (\mathcal{X}_0)_h\big) h^{\beta_0-C\alpha_0\int_0^t W(\tau)d\tau}\times e^{Ct\|\nabla\theta_0\|_{L^\infty}}r^{-C\int_0^t W(\tau)d\tau}).\\
\end{eqnarray*}
By exploring the previous estimate with identity (\ref{Eq32}) stated in Definition \ref{D7}, we deduce that
\begin{equation}\label{eq42}
\|\omega(t)\|^{\epsilon}_{(\Sigma(t))_{\delta ^{-1}_t(h)} , (\mathcal{X}_t)_h} \leq C_0 e^{Ct}h^{\beta_0 - C\int_0^t W(\tau)d\tau} e^{Ct\|\nabla\theta_0\|_{L^\infty}}r^{-C\int_0^t W(\tau)d\tau}.
\end{equation}
Now, we are willing to apply the logarithmic estimate in Theorem \ref{thm3}, Proposition \ref{prop5} and the monotonicity of the map
$ g\mapsto g\Log\big(e + \frac{a}{g}\big)$, we discover that
$$ \|\nabla v(t)\|_{L^\infty((\Sigma_t)^c_{\delta^{-1}_t(h)}} \leq C \bigg( \|\omega_0\|_{L^1 \cap L^\infty} +t\|\nabla\theta_0\|_{L^a \cap L^\infty}r^{-C\int_0^t W(\tau)d\tau} \bigg)\Log\Bigg(e+\frac{\|\omega(t)\|^{\epsilon}_{(\Sigma(t))_{\delta ^{-1}_t(h)} , (\mathcal{X}_t)_h}}{\|\omega\|_{L^\infty}}\Bigg). $$
From (\ref{eq42}), one obtains
\begin{eqnarray*}
\|\nabla v(t)\|_{L^\infty((\Sigma_t)^c_{\delta^{-1}_t(h)}} &\leq & C \bigg( \|\omega_0\|_{L^1 \cap L^\infty} +t\|\nabla\theta_0\|_{L^a \cap L^\infty}r^{-C\int_0^t W(\tau)d\tau} \bigg)\\
& \times & \bigg(M_0+t+t\|\nabla\theta_0\|_{L^a \cap L^\infty}r^{-C\int_0^t W(\tau)d\tau} + \big(\beta_0 - C\int_0^t W(\tau)d\tau\big)\Log h \bigg).
\end{eqnarray*}
As a consequence we have
\begin{eqnarray*}
\frac{\|\nabla v(t)\|_{L^\infty((\Sigma_t)^c_{\delta^{-1}_t(h)}}}{-\Log \delta^{-1}_t(h)} &\leq & C \bigg( \|\omega_0\|_{L^1 \cap L^\infty} +t\|\nabla\theta_0\|_{L^a \cap L^\infty}r^{-C\int_0^t W(\tau)d\tau} \bigg)\\
&\times & \bigg(M_0+t+t\|\nabla\theta_0\|_{L^a \cap L^\infty}r^{-C\int_0^t W(\tau)d\tau} + \int_0^t W(\tau)d\tau  \bigg)\Exp\Big(\int_0^t \|\nabla v(\tau)\|_{LL}d\tau  \Big).
\end{eqnarray*}
With the help of the definition $W(\tau)$ stated in the proposition \ref{prop4}, one has
\begin{eqnarray*}
W(t) &\leq & C \bigg( \|\omega_0\|_{L^a \cap L^\infty} +t\|\nabla\theta_0\|_{L^a \cap L^\infty}r^{-C\int_0^t W(\tau)d\tau} \bigg)\\
&\times & \bigg(M_0+t+t\|\nabla\theta_0\|_{L^a \cap L^\infty}r^{-C\int_0^t W(\tau)d\tau} + \int_0^t W(\tau)d\tau  \bigg)\Exp\big(2\int_0^t \|\nabla v(\tau)\|_{LL}d\tau  \big).
\end{eqnarray*}
We pick $T>0$ such that
\begin{equation}\label{eq43}
T\|\nabla\theta_0\|_{L^a \cap L^\infty}r^{-C\int_0^T W(\tau)d\tau} \leq \min(1,\|\omega_0\|_{L^1 \cap L^\infty}).
\end{equation}
Proposition \ref{prop5} combined with Lemma \ref{lem5} gives for all $t\in[0,T]$
\begin{eqnarray*}
 \| v(t) \|_{LL} &\leq & \|\omega(t)\|_{L^a \cap L^\infty} \\
&\leq & \|\omega_0\|_{L^a \cap L^\infty} + t \|\nabla\theta_0\|_{L^a \cap L^\infty}r^{-C\int_0^T W(\tau)d\tau} \\
&\leq & 2\|\omega_0\|_{L^a \cap L^\infty}.
\end{eqnarray*}
Accordingly, we have
$$W(t) \leq C\|\omega_0\|_{L^a \cap L^\infty}\big(M_0+t+\int_0^T W(\tau)d\tau\big)e^{Ct\|\omega_0\|_{L^a \cap L^\infty}}.$$
Gronwall's inequality ensures that for $t\in[0,T]$
\begin{equation}\label{W-b}
W(t) \leq C\|\omega_0\|_{L^a \cap L^\infty}(M_0+t)e^{Ct\|\omega_0\|_{L^a \cap L^\infty}}e^{e^{Ct\|\omega_0\|_{L^a \cap L^\infty}}}.
\end{equation}
Finally, we gain two principal estimate
\begin{equation*}
\int_0^T W(\tau)d\tau \leq (M_0+t)e^{e^{Ct\|\omega_0\|_{L^a \cap L^\infty}}},\quad  r^{-C\int_0^T W(\tau)d\tau} \leq r^{-(M_0+t)e^{e^{Ct\|\omega_0\|_{L^a \cap L^\infty}}}}.
\end{equation*}
In order to satisfy the assumption (\ref{eq43}) it is enough to take

$$ T\|\nabla\theta_0\|_{L^\infty}r^{-(M_0+t)e^{e^{Ct\|\omega_0\|_{L^a \cap L^\infty}}}} = \min\big(1,\|\omega_0\|_{L^1\cap L^\infty}).$$
So, the continuity process confirms the existence of such $T>0$. This achieves the proof.
\end{proof}
\subsection{Existence and uniqueness} We embark by mollifying the initial data, by setting $v_{0,n}=S_nv_0, \omega_{0,n}=S_n\omega$ and $G^{n}_i(\theta_0) = \rho_n \star G_i(\theta_0), i\in\lbrace 1,2 \rbrace$, with $S_n$ is a cut-off operator already defined in subsection \ref{Little-P} and $\rho_n(x)=n^2\rho(nx)$, with $\rho\in\mathscr{D}(\RR^2)$ be a positive function supported in unit ball and satisfying $\int_{\RR^2} \rho(x)dx =1.$ The quantity $G^n _i(\theta_0)$ being a constant in a small neighborhood of $\Sigma_0$. Indeed, we consider the set $$(\Sigma_0)_{r-\frac1n} \triangleq \lbrace x \in \RR^2, d(x,\Sigma_0) \leq r-1/n\rbrace.$$
On the other hand for $n$ big enough, we have $(\Sigma_0)_{\frac{r}{2}}\subset (\Sigma_0)_{r-\frac1n}$. Then for $x\notin (\Sigma_0)_{r-\frac1n}$, then $x\notin(\Sigma_0)_{\frac{r}{2}}$, which implies $G_i(\theta_0)=0$. Consequently $G^n _i(\theta_0)=\rho_n\star G_i(\theta_0)=0$ which gives the result.  

\hspace{0.5cm}For the uniformness bound for $\|\omega_{0,n}\|_{L^{a} \cap L^{\infty}},\; \|\nabla G^n_i(\theta_{0})\|_{L^{a} \cap L^{\infty}},\; \|\omega_{0,n}\|^{\epsilon}_{X_0}$, we explore the properties of mollifier argument and Young's inequality to obtain

$$\|\rho_n \star G_i(\theta_0) \|_{L^{a}} \leq \|\theta_0 \|_{L^{a}}, \qquad \|\rho_n \star    \nabla G_i(\theta_0) \|_{L^{a}\cap L^{\infty}}\leq \|\nabla\theta_0 \|_{L^{a}\cap L^{\infty}}.$$
and
$$\|\omega_{0,n}\|_{L^{a} \cap L^{\infty}} \leq \|\omega_{0}\|_{L^{a} \cap L^{\infty}},\quad \|\omega_{0,n}\|^{\epsilon}_{X_0} \leq \|\omega_{0}\|^{\epsilon}_{X_0}.
$$
For the term $\partial_{X_{0,\lambda}} G^n_i(\theta_{0})$, we write

$$\|\partial_{X_{0,\lambda}} G^n_i(\theta_{0})\|_{\epsilon} = \|\rho_n\star(\partial_{X_{0,\lambda}}G_i (\theta_{0}))\|_{\EE}  +  \big[\partial_{X_{0,\lambda}},\rho_n\star]G_i (\theta_{0})\|_{\epsilon}.$$

We make use the fact $X_{0,\lambda}\in C^\EE$ and $G_i(\theta_0)\in W^{1,\infty}(\RR^2)$ we exploit the following result \cite{Hassainia-Hmidi}
$$
\|[\partial_{X_{0,\lambda}},\rho_n\star]G_i (\theta_{0})\|_{\epsilon}\lesssim\|X\|_{\EE}\|\nabla(G_i(\theta_0))\|_{L^\infty}.
$$
to conclude
$$\|\partial_{X_{0,\lambda}} G^n_i(\theta_{0})\|_{\epsilon} \leq \||\partial_{X_{0,\lambda}}\theta_{0}\|_{\epsilon} + \check{\|}X_{0,\lambda}\|_{\epsilon}\|\nabla \theta_0\|_{L^{\infty}}.$$
We follow closely the same steps presented in the existence results in the case of the regular patch, we achieve the result.

\hspace{0.5cm}We will now prove uniqueness of solutions in the space $\mathcal{M}=L^{\infty}([0,T],L^2)$. In this part the uniqueness issue doesn't similar o the formalism of regular patches because the velocity is not Lipschitzian everywhere it belongs to $LL-$space. For this purpose, we follow Yudovich's approach. To do this, let $(v_0,\theta_0)$ be a smooth solution belonging to $(b+L^2)\times L^2,$ where $b$ is a stationary vector field in the sense
$$b(x)=\frac{x^{\bot}}{|x|^2} \int_0^{|x|} sf(s)ds, $$
where $f \in \mathcal{D}(\RR^2)$ supported away from the origin. Then any local solution $\big(v(t),\theta(t)\big)$ of the system \eqref{Eq1} belongs to $(b+L^2)\times L^2$. 

\hspace{0.5cm}Without loss of generality setting  $ b=0$ and let $(v_i,\nabla p_i,\theta_i) \in \mathcal{M}, 1\leq i \leq 2$ be two solutions of the system \eqref{Eq1} and denote $\delta v = v_1 - v_2,\;\delta p = p_1 - p_2$ and $\delta G_{i}(\theta) = G_i(\theta_1) - G_i(\theta_2)$ then a straightforward computations claim that $(\delta v,\delta p,\delta\theta)$ evolves
\begin{equation}\label{Diff} 
\left\{ \begin{array}{ll}
\partial_{t} \delta v+v_2\cdot\nabla \delta v = G(\theta_1)-G(\theta_2)-\nabla \delta p - \delta v\cdot\nabla v_1,& \\
 \partial_{t} (\delta G_i(\theta))+v_2\cdot\nabla\delta (G_i(\theta))= -\delta v\cdot\nabla(G_i(\theta)), & \\
\Div v=0, &\\
({v},{\theta})_{| t=0}=({v}_0,{\theta}_0).
\end{array} \right.
\end{equation}
A standard $L^2$ estimate for \eqref{Diff} combined with H\"older inequatlity gives for $q\in [a,\infty[ $ with the notation $q'=\frac{q}{q-1}$ that
\begin{equation}\label{Diff-v-theta}
\left\{\begin{array}{ll}
\frac{d}{dt} \|\delta v(t)\|^{2}_{L^2} \leq 2\|\nabla v_1(t)\|_{L^q} \|\delta v(t)\|^{2}_{L^{2q'}} + 2\|G(\theta_1)-G(\theta_2)\|_{L^2}\|\delta v(t)\|_{L^2}, &\\
\frac{d}{dt}\|\delta G_i(\theta)(t)\|^{2}_{L^2} \leq 2\|\nabla (G_i(\theta_1)(t))\|_{L^\infty}\|\delta v(t)\|_{L^2}\|\delta G_i(\theta)(t)\|_{L^2}.
\end{array}
\right.
\end{equation}
By interpolation, \eqref{Diff-v-theta} takes the form
\begin{equation}\label{Diff-v-theta-1}
\left\{\begin{array}{ll}
\frac{d}{dt} \|\delta v(t)\|^{2}_{L^2} \lesssim  q\|\nabla v_1\|_{L}\|\delta v(t)\|^{\frac{2}{q}}_{L^{\infty}}\|\delta v(t)\|^{\frac{2}{q'}}_{L^{2}} +  (\|\delta G_1(\theta)\|_{L^{2}} + \|\delta G_2(\theta)\|_{L^{2}} ) \|\delta v(t)\|_{L^2},&\\
\frac{d}{dt}\|\delta G_i(\theta(t))\|^{2}_{L^2} \leq 2\|\nabla G_i(\theta_1(t))\|_{L^\infty}\|\delta v(t)\|_{L^2}\|\delta G_i(\theta(t))\|_{L^2},\quad i=1,2.
\end{array}
\right.
\end{equation}

with 
$$
\|\nabla v_1\|_{L}\triangleq\sup_{q\in[2,\infty[}\frac{\|\nabla v_1\|_{L^q}}{q}.
$$
Since $\omega_0\in L^a\cap L^\infty$ then \eqref{W-b} implies that the function $W$ is locally bounded. With the aid of \eqref{prop5:2}, one deduce that $\|\nabla v_1\|_{L}$ is also locally bounded. On the other hand, for $i\in\{1,2\}$ as aforementioned above $v_i\in L^{\infty}_t L^2$ and $\omega_i\in L^\infty_t L^\infty$  provide that $\delta v\in L^\infty_t L^\infty$. Meaning that the r.h.s. of $\delta v$ estimate in \eqref{Diff-v-theta-1} is well-defined.

\hspace{0.5cm} Next, for $n \in \NN_{>0}$, take
$$ \mathfrak{E}_n(t) \triangleq \sqrt{\sum_{i=1}^2\|\delta G_i(\theta) (t)\|^{2}_{L^2}+\|\delta v(t)\|^{2}_{L^2}+\frac{1}{n}}.$$
By a straighforward calculations, we get 
$$\frac{d}{dt}\mathfrak{E}_n(t) \leq Cq \|\nabla v_1\|_{L}\|\delta v(t)\|^{\frac{2}{q}}_{L^{\infty}} {\mathfrak{E}_n(t)}^{1-\frac{2}{q}} +\lambda(t)\mathfrak{E}_n(t).$$
with $\lambda(t)=\big(1+\sum_{i=1}^2\|\nabla \big(G_i(\theta_1)\big)(t)\|_{L^\infty}\big)$. By setting $\mathfrak{F}_n(t) = e^{-\int_0^t\lambda (\tau)d\tau}\mathfrak{E}_n(t)$. Also, we have
\begin{equation}\label{F-n}
\frac{d}{dt}\mathfrak{F}_n(t) {\mathfrak{F}_n(t)}^{\frac{2}{q}-1}\leq Cq \|\nabla v_1\|_{L}\|\delta v(t)\|^{\frac{2}{q}}_{L^{\infty}}e^{-\frac2q\int_0^t\lambda (\tau)d\tau}.
\end{equation}
Let us denote that function $\gamma(t)=-\int_{0}^{t}\lambda(\tau)d\tau$ represents the loss of regularity in the process of $LL$ class. Consequently, by developing a time integration one deduce

$$
{\mathfrak{F}_n(t)} \leq \bigg(\Big(\frac{1}{n}\Big)^{\frac{2}{q}}+ C\int_0^t \|\nabla v_1(\tau)\|_{L}\|\delta v(\tau)\|^{\frac{2}{q}}_{L^{\infty}}d\tau\bigg)^{\frac{q}{2}}.$$

Letting now $n$ goes to infinity to obtain for $t>0$
\begin{equation}\label{delta}
\|\delta v(t)\|^{2}_{L^2}+\sum_{i=1}^{2}\|\delta G_i(\theta) (t)\|^{2}_{L^2} \leq \|\delta v(t)\|^{2}_{L^\infty_t L^\infty}\bigg( C\int_0^t \|\nabla v_1(\tau)\|_{L}d\tau\bigg)^q.
\end{equation}
 
The finitude of the quantity $\|\nabla v_1\|_{L}$ ensures the extistence of $t^{\star}$ fullfils
$$
C\int_0^{t^\star} \|\nabla v_1(\tau)\|_{L}d\tau <1.
$$
Letting $q$ goes to infinity in \eqref{delta}, we find that $\delta v=\delta G_i(\theta)=0$ on $[0,t^\star]$. In accordance to the connectivity argument, one may conclude that $\delta v=\delta\theta=0$ on $[0,t]$ for all $0 \leq t \leq T$ which implies that $\delta \theta=0$. Indeed, as $\theta_i, 1\le i\le2$ is transported by the flow then $\theta_i(t,x)=\theta_0(\Psi_{v_i}^{-1}(t,x))$, or $v_1=v_2$ meaning that $\Psi_{v_1}^{-1}\equiv\Psi_{v_2}^{-1}$. Finally, we infer that $\theta_1=\theta_2$ and this ends the proof of uniquness part.

\subsection{Proof of Theorem \ref{thm2}}  By hypothesis $D_0$ is an open bounded domain with a boundary is a Jordan curve of H\"older regularity $C^{\epsilon+1}$ outside $\Sigma_0$. The geometric boundary of $D_0$ provides in view of Definition \ref{Chart} a real function $f_0\in C^{1+\epsilon}$ such that $\partial D_0=f^{-1}_0(\lbrace0\rbrace)\cap V_0$ and $\nabla f_0 \neq 0$ on $V_0\setminus\Sigma_0.$ We assume there exists a real number $\text{ }\gamma^{'}_0>0,\text{ }$ such that
\begin{equation}\label{Cond-bound}
|\nabla f_0(x)| \geq d(x,\Sigma_0)^{\gamma'_0}, \quad  \forall x\in V_0.
\end{equation}
Such condition is imposed to telle you that the curves constitute the boundary of $D_0$ are not tangent to one another at infinite order at the singular points. On the other hand, let $(\vartheta_h)_{h\in ]0,e^{-1}]}$ be an indexed family such that $\vartheta\in\mathcal{D}(\RR^2),$ 
$$ \supp \vartheta_h \subset (\Sigma_0)^{c}_{\frac{h}{2}},\quad  \vartheta(x) = 1,\quad \forall x \in (\Sigma_0)^{c}_{h},$$
and satisfying
$$ \|\vartheta_h\|_r \leq C(r)\frac{1}{h^r}, \quad \forall (r,h)\in \RR_+\times]0,e^{-1}].$$
On the other hand, for $\varphi$ be a smooth function be such that $ \supp \varphi \subset V_0,  \varphi(x)=1$ for every $x\in V_1,$
with $V_1$ is a small nighbrohood of $V_0.$

\hspace{0.5cm}Let us introduce the family $(\mathcal{X}_{0}=X_{0,\lambda,h})_{(\lambda,h)\in\lbrace0,1\rbrace\times]0,e^{-1}]}$, with
$$ X_{0,0,h} = \nabla^{\bot}(\vartheta_h f_0), \quad X_{0,1,h} = (1-\varphi)\left( \begin{array}{ll}
0 \\
1
\end{array} \right).$$
We see if $\mathcal{X}_{0}$ is $\Sigma_0$-admissible of order $\Theta_0=(\alpha_0,\beta_0,\gamma_0)$. Clearly, $X_{0,0,h}\in C^\EE$ because $f\in C^{1+\EE}$ and $\Div X_{0,0,h}=0$, while $X_{0,1,h}\in C^{\infty}$. Also, by construction for some $\alpha_0>1,\;\supp X_{0,i,h} \subset (\Sigma_0)^{c}_{h/2} \subset (\Sigma_0)^{c}_{h^{\alpha_{0}}}$ and in light of \eqref{Cond-bound} we may choose $\gamma_0=-\gamma'_0$ to conclude that
$$\check{\|}X_{0,i,h}\|_\epsilon \lesssim \frac{1}{h^{1+\epsilon}}.$$
Therefore, it is enough to take $\beta_0=\gamma_0-\epsilon-1$ to obtain the order $\Theta_0.$ On theother hand, we write

$$\partial_{X_{0,0,h}} \omega_0 = \vartheta_h \partial_{\nabla^\bot f_0} \omega_0 + f_0\partial_{\nabla^\bot\vartheta_h }\omega_0.      $$
It is easy to verify that $\partial_{\nabla^\bot\vartheta_h } \omega_0\in \mathscr{D}'(\RR^2)$ of order $0$ and $\supp (\partial_{\nabla^\bot\vartheta_h } \omega_0) \subset \partial D_0$, so that $f_0\equiv0$ over $\partial D_0$ leading to $\partial_{X_{0,0,h}} \omega_0 = 0.$ Whereas, $\partial_{X_{0,1,h}}\omega_0=0$ follows from the fact $1-\varphi$ vanishes on $W_0\subset\partial D_0$. Finally, we claim the regularity of the an initial density $\theta_0$. Doing so, we make use to the fact $G_i(\theta_0),\; 1\le i\le 2$ is constant in neighborhood of $\Sigma_0$, so that $\nabla G_i(\theta_0)=0$ in the same neighborhood. It follows that $\nabla^\perp\vartheta_h\cdot\nabla G_i(\theta_0)=G'_{i}(\theta_0)\nabla\theta_0\cdot\nabla^\perp\vartheta_h=0$. Besides, $\partial_{\nabla^\perp}\theta_0$ belongs to $C^\EE$. Indeed, the assumptions $\theta_0\in \Lip$ and $f_0\in C^{1+\EE}$ yield in first time that $\nabla^\bot f_0\in C^\epsilon$ and
\begin{eqnarray*}
\big\|\partial_{\nabla^\perp f_0}\theta_0\big\|_{C^\EE}&=&\big\|\nabla^\perp f_0\cdot\nabla\theta_0\big\|_{C^\EE}\\
&\le&\|\nabla^\perp f_0\|_{C^\EE}\|\theta_0\|_{\Lip}<\infty.
\end{eqnarray*}
This gives $\partial_{\nabla^\perp f_0}\theta_0\in C^\EE$. For the term $\partial_{X_{0,0,h}}\theta_0$, we make use $\vartheta_h\in\mathscr{D}(\RR^2)$ and $\partial_{\nabla^\perp}\theta_0\in C^{\EE}$, we immediately deduce that $\partial_{X_{0,0,h}}\theta_0\in C^{\EE}$, so the assumptions of Theorem \ref{Thm2-gen} are fulfilled in the sense that the local well-posedness of Theorem \ref{thm2} is now achieved. To guaranty the regularity of the boundary $\partial D_t$ outside of $\Sigma_t$ we explore the same scenario as in the regular patch case.

\end{document}